\documentclass[11pt,reqno]{amsart}

\font\sml=cmr6

\newcommand{\K}{{\mathbf K}^\crit}

\newcommand{\sm}{\setminus}

\newcommand{\kahler}{K\"ahler }

\newcommand{\RR}{\mathbb{R}}

\newcommand{\CC}{{\mathbb C}}

\newcommand{\ZZ}{{\mathbb Z}}

\newcommand{\R}{{\mathbb R}}

\newcommand{\dbar}{\bar\partial}
\newcommand{\ddbar}{\partial\dbar}

\newcommand{\acal}{\mathcal{A}}

\newcommand{\ccal}{\mathcal{C}}

\newcommand{\fcal}{\mathcal{F}}

\newcommand{\hcal}{\mathcal{H}}

\newcommand{\lcal}{\mathcal{L}}
\newcommand{\mcal}{\mathcal{M}}
\newcommand{\ncal}{\mathcal{N}}

\newcommand{\scal}{\mathcal{S}}

\def    \id {{\operatorname{id}}}

\def    \span   {{\operatorname{span}}}

\def    \open   {{\operatorname{o}}}
\def    \span   {{\operatorname{span}}}

\def    \cvx   {{\operatorname{cvx}}}
\def    \sign   {{\operatorname{sign}}}

\def    \R  {{\mathbb R}}

 \def    \Im     {{\operatorname{Im}}}

\newtheorem{theo}{{\sc Theorem}}[section]
\newtheorem{thm}[theo]{{\sc Theorem}}
\newtheorem{cor}[theo]{{\sc Corollary}}

\newtheorem{remark}[theo]{{\sc Remark}}
\newtheorem{lem}[theo]{{\sc Lemma}}
\newtheorem{prop}[theo]{{\sc Proposition}}
\newtheorem{definition}[theo]{{\sc Definition}}

\newenvironment{defin-no-number}{\medskip\noindent{\it Definition:\/} }{\medskip}

\def\samename{\underline{\phantom{aaaaaaaa}}$\,$}

\def\h#1{\hbox{#1}}

\def\o{\omega}

\def\MA{Monge--Amp\`ere }

\def\K{K\"ahler }

\def\ra{\rightarrow}
\def\a{\alpha}
\def\th{\theta}
\def\vp{\varphi}

\def\isom{\cong}

\def\w{\wedge}
\def\i{\sqrt{-1}}
\def\text{\textstyle}

\def\ra{\rightarrow}

\def\isom{\cong}

\def\del{\partial}

\def\dis{\displaystyle}

\def\Tspancvx{T_{\span}^\cvx}

\def\calM{\mcal}
\def\calA{\acal}
\def\calC{\ccal}
\def\calL{\lcal}

\def\calS{\scal}
\def\calF{\fcal}
\def\calH{\hcal}

\def\bigs{\bigskip}

\def\eps{\epsilon}

\newcommand{\Hilb}{{{\operatorname{Hilb}}}}

\def\gM{g_{\hbox{\sml M}}}

\def\longra#1{\hbox to #1pt{\rightarrowfill}}
\def\longla#1{\hbox to #1pt{\leftarrowfill}}

\def\mapright#1#2{\smash{
                  \mathop{\longra{#2}}\limits^{#1}}}

\def\mapleft#1#2{\smash{
                  \mathop{\longla{#2}}\limits^{#1}}}

\def\mapdown#1{\Bigg\downarrow
                 \rlap{$\vcenter{\hbox{$\scriptstyle#1$}}$}}

\title[
The Cauchy problem for the Monge--Amp\`ere equation III]
{The Cauchy problem for the homogeneous Monge--Amp\`ere equation, 
III. Lifespan}

\author[Y.A. Rubinstein]{Yanir A. Rubinstein }
\address{Department of Mathematics, Stanford University, Stanford, CA 94305, USA}
\email{yanir@member.ams.org}

\author[S. Zelditch]{Steve Zelditch }
\address{Department of Mathematics, Northwestern  University, Evanston, IL 60208, USA} 
\email{zelditch@math.northwestern.edu}

\thanks{\hglue-10pt May 21, 2012}

\begin{document}

\begin{abstract}

We prove several results on the lifespan, regularity, and uniqueness of solutions of the Cauchy
problem for the homogeneous complex and real \MA equations (HCMA/HRMA) under various a priori regularity conditions.
We use methods
of characteristics in both the real and complex settings to bound the lifespan
of  solutions with prescribed regularity.  In the complex domain,
we characterize the $C^3$ lifespan of the HCMA in terms of analytic
continuation of Hamiltonian mechanics and intersection of complex time
characteristics. 
We use a conservation law type argument to prove uniqueness
of solutions of the Cauchy problem for the HCMA. 
We then prove that the Cauchy problem is ill-posed in $C^3$, in the sense 
that there exists a dense set of $C^3$ Cauchy data for which there exists 
no $C^3$ solution even for a short time. 
In the real domain we show that the HRMA is equivalent
to a Hamilton--Jacobi equation, and use the equivalence to prove that 
any differentiable weak solution is smooth,
so that the differentiable lifespan equals the convex lifespan determined 
in our previous articles.
We further  show that the only obstruction
to $C^1$ solvability is the invertibility of the associated Moser maps. 
Thus, a smooth solution of the Cauchy problem for HRMA exists for a positive
but generally finite time and cannot be continued even as a weak $C^1$ solution afterwards.
Finally, we introduce the notion of a ``leafwise subsolution" for the HCMA that
generalizes that of a solution, and many of our aforementioned results are proved
for this more general object.

\end{abstract}

\maketitle


\section{Introduction}
\bigs

This article is the third in a series \cite{RZ1,RZ2}  
whose aim is to study existence, uniqueness, and regularity of
solutions of the initial 
value problem (IVP) for geodesics in the space
\begin{equation}
\label{HoEq}
\textstyle\calH_\o
=
\{\vp\in C^{\infty}(M) \,:\, \omega_\vp:= \omega+\i\ddbar\vp>0\}
\end{equation}
of \kahler metrics on a compact \K manifold $(M, \omega)$
in the  class of $\omega$, where $\calH_\o$ is equipped with the
metric \cite{M,S,D1}
$$
\gM(\zeta,\eta)_{\vp}:= \int_M
\zeta\eta\, {\omega_{\vp}^m},\quad \vp \in \hcal_{\o},\quad
\zeta,\eta \in T_{\vp} \hcal_{\o}\isom C^\infty(M).
$$
This initial value problem is a special case of the Cauchy problem for
the homogeneous complex/real Monge--Amp\`ere equation (HCMA/HRMA).  
The IVP is long believed to be ill-posed, and a  motivating  problem is to prove
that this is indeed the case, to determine
which initial data give rise to solutions, especially those
of relevance in geometry (`geodesic rays'), to construct the
solutions,  and to determine the lifespan $T_\span$ of solutions
for general initial data.  

In this article, we prove a number of results on the lifespan, regularity, and uniqueness of solutions of the Cauchy
problem for the HCMA and the HRMA equations under various a priori regularity conditions.
The results  are  based on a study of the `characteristics' of the HCMA/HRMA equations,
or more precisely on the relations between solutions of 
these equations
and Hamiltonian mechanics,  and to solutions of related Hamilton--Jacobi equations. 

First, we characterize the $C^3$ lifespan of the HCMA and prove uniqueness
of classical solutions. We then introduce the notion
of a leafwise subsolution of the HCMA that generalizes the notion of a solution, 
and derive obstructions to its existence. This can be considered as a method
of `complex characteristics'.
Combining these results we establish that the IVP for the HCMA is locally ill-posed in $C^3$.
This puts a restriction on Cauchy data, and addresses questions about
the Cauchy problem raised by the work of Mabuchi, Semmes, and Donaldson 
\cite[p. 238]{M},\cite{S},\cite[p. 27]{D1}. 
We then study the notion of a leafwise subsolution for the HRMA,
and prove its uniqueness. This allows us to characterize the Legendre
transform subsolution of the prequels \cite{RZ1,RZ2}, and determine the $C^1$ lifespan
of the HRMA. A key ingredient here
is an apparently new connection between HRMA and Hamilton--Jacobi equations. 

\subsection{Obstructions to solvability, uniqueness, and the smooth lifespan of the HCMA}

We begin in the complex domain, where 
 Semmes and Donaldson \cite{S,D1}  gave a  formal solution of the IVP 
in terms of holomorphic characteristics.
Namely, the  Cauchy data $(\omega_{\vp_0}, \dot{\vp}_0)$ of the IVP determines a Hamiltonian flow $\exp t X_{\dot{\vp}_0}^{\omega_{\vp_0}}$.
If the orbits $\exp t X_{\dot{\vp}_0}^{\omega_{\vp_0}} z$ of the flow admit analytic continuations in time up
to imaginary time $T$, one obtains
a family of maps 
\begin{equation} \label
{MAPS} 
f_\tau(z) = \exp -\sqrt{-1} \tau  X_{\dot{\vp}_0}^{\omega_{\vp_0}} z: S_T \times M \to M,
\end{equation}
where 
$$
S_T=[0,T]\times\RR,
$$
with $\tau=s+\i t\in S_T$, $s\in[0,T]$ and $t\in\RR$.
The formal solution $\vp_s$ is then given by the formula,
\begin{equation} \label
{FORMAL}
(f_s^{-1})^\star  \omega_{\vp_0} - \omega_{\vp_0} = \i \ddbar \vp_s, \quad s\in[0,T].
\end{equation}

There are several obstructions to solving the IVP in this manner, which must vanish if there exists a
$C^3$ solution.  The most obvious one is that the Hamilton orbits need
to possess  analytic continuations to a strip $S_T$. 
This analytic extension problem for orbits
should already be an ill-posed problem, and we say that the Cauchy data is 
``$T$-good" if the extension exists and $f_s$ is smooth (see Definitions 
\ref{TGoodDef}--\ref{TGoodDef}).   
This is a Cauchy problem for a holomorphic map into a nonlinear space, and we do 
not study it directly here; but in \S \ref{1.2} we describe some results on obstructions
to closely related linear Cauchy problems.

In several settings, such 
as torus-invariant Cauchy data on toric varieties, the Hamilton orbits for smooth
Cauchy data do
possess analytic continuations (see Proposition \ref{ToricLifespanProp} below). 
As the following theorem shows, the only additional obstruction
to solving the HCMA smoothly is that the space-time complex Hamilton orbits may intersect.   
To state the result precisely,
let $(M,J,\o)$ be  a compact closed connected \K manifold 
of complex dimension $n$. 
The IVP for geodesics is equivalent to the following
Cauchy problem for the HCMA 
\begin{equation}
\label{HCMARayEq}
\begin{aligned}
(\pi_2^\star\omega + \i\ddbar \vp)^{n+1}
=
0,
\quad 
(\pi_2^\star\omega + \i\ddbar \vp)^{n}
\ne
0,
\;\; &\mbox{on} \; S_{T} \times M,
\cr
\vp(0,t,\,\cdot\,)
 =
\vp_0(\,\cdot\,), \quad
\partial_s\vp(0,t,\,\cdot\,)
=
\dot\vp_0(\,\cdot\,), \;\; &\mbox{on} \; \{0\}\times\RR \times M.
\end{aligned}
\end{equation}
\noindent
where $\pi_2: S_{T} \times M \to M$ is the projection, and
where $\vp$ is 
is required to be $\pi_2^\star\o$-plurisubharmonic (psh) on $S_T\times M$.
The rest of the notions in the following theorem are
defined in \S\ref{HamFlowsHCMASection}.

\begin{thm}
\label
{HCMACauchyCthreeThm} {\rm (Smooth lifespan and uniqueness)}
Let $(M,\omega_{\vp_0})$ be a compact \K manifold.
The Cauchy problem (\ref{HCMARayEq}) with $\omega_{\vp_0}\in C^1$
and $\dot\vp_0\in C^3(M)$ has a 
solution in $C^3(S_T \times M)\cap PSH(S_T \times M,\pi_2^\star\o)$  if and only if the Cauchy data is $T$-good
and the maps $f_s$ defined by \eqref{MAPS}  are $C^1$ and admit a $C^1$ 
inverse for each $s\in[0,T]$. The solution is unique in
$C^3(S_T \times M)\cap PSH(S_T\times M,\pi_2^\star\o)$.

\end{thm}

This result is important in clarifying the nature of the obstructions to solving the HCMA.
The existence proof follows by a rather straightforward combination of the Semmes--Donaldson arguments \cite{S,D1}.
The uniqueness proof, somewhat surprisingly, does not readily adapt from the $\CC^n$ setting studied
by Bedford--Burns \cite{BB}. 
Unlike in their setting, the proof is not local in nature, 
and requires a global conservation
law type argument. The key difference is that the stripwise equations vary from leaf to leaf,
and one has to prove an a priori estimate that ensures that the stripwise elliptic problems
are not degenerating. 
The uniqueness proof is also completely different from the corresponding proof 
for the Dirichlet problem, where the maximum principle is available.

Henceforth, we describe breakdown in time of solutions in terms of 
lifespans. 
\begin{definition}
Let the $C^{k,\alpha}$ lifespan $T^{k,\alpha}_\span$ (respectively, lifespan
$T_\span$) of the Cauchy problem
(\ref{HCMARayEq}) be the supremum over all $T\ge 0$ such that
(\ref{HCMARayEq}) admits a solution in $C^{k,\alpha}(S_T \times M)\cap PSH(S_T\times M,\pi_2^\star\o)$ 
(respectively, in  $PSH(S_T\times M,\pi_2^\star\o))$.
\end{definition}

We thus have the following characterization of the smooth lifespan of the HCMA.
The same result holds also for the $C^3$ lifespan $T^3_\span$.

\begin{cor}
\label{HCMACauchyLifespanCor} %

The smooth lifespan $T^\infty_\span$ of the Cauchy problem
(\ref{HCMARayEq}) with smooth initial data is the supremum over
$T\ge 0$ such that the Cauchy problem is 
$T$-good and the maps $f_s$ defined by (\ref{FsMapEq}) 
are smoothly invertible for each $s\in[0,T]$.

\end{cor}

\subsection{\label{1.2} Leafwise subsolutions for the HCMA and ill-posedness}

In the apparent absence of weak solutions beyond the convex lifespan, 
motivated by the detailled results of the prequel \cite{RZ2} on the Legendre
subsolution in the special case of the HRMA, we are led to introduce a notion 
of a {\it leafwise subsolution}, that should be
an ``optimal" subsolution in some situations.

\begin{definition}
\label{OptimalSubsolutionDef}
Assume that the Cauchy problem (\ref{HCMARayEq}) is $T$-good.
We call a $\pi_2^\star\o$-psh function $\vp$ on $S_T\times M$ a
{\rm $T$-leafwise subsolution} of the HCMA (\ref{HCMARayEq})
if it satisfies the initial conditions of the Cauchy problem (\ref{HCMARayEq}),
if $(\pi_2^\star\o+\i\ddbar\vp)^n\ne0$,
and if for each $z\in M$, we have
\begin{equation}
\label
{LeafwiseDefEq}
\gamma_z^\star(\pi_2^\star\o+\i\ddbar\vp)=0,
\end{equation}
where $\gamma_z(\tau)=(\tau,  f_{\tau}(z))$, and
$f_\tau(z):= \exp-\i \tau
X_{\dot\vp_0}^{\omega_{\vp_0}}.z$.
\end{definition}

The proof of Theorem \ref{HCMACauchyCthreeThm} shows that a $C^3$ solution 
on $[0,T]\times M$ is a $T$-leafwise subsolution, but ``leafwise subsolutions" 
are more general: a subsolution of HCMA may solve \eqref{LeafwiseDefEq}
along leaves without solving the HCMA globally since the invertibility condition 
on $f_s$ in Theorem \ref{HCMACauchyCthreeThm} may fail, e.g., when the
leaves intersect.

One of our main results is that  the problem of existence of a leafwise subsolution
for the Cauchy problem for the HCMA
is already locally ill-posed in time.
In particular, this implies the ill-posedness in $C^3$ of the Cauchy
problem for the HCMA itself.

\begin{thm} \label{GENERIC} {\rm (Local ill-posedness)}
For each 
${\vp_0} \in \hcal_{\omega}$  
there exists a dense set of $\dot\vp_0\in C^3(M)$ 
for which  $T_\span^\infty=T^3_\span=0$,
i.e., the IVP \eqref{HCMARayEq} admits no $C^3$ solution for any $T>0$.
\end{thm}

The proof is given in Section  \ref{LeafwiseSection}. 
The Cauchy problem
for the HCMA is a multi-dimensional, nonlinear generalization of
the Cauchy problem for the Laplace equation 
 $\Delta u = 0$ on each strip $S_T$, which is 
one of the classic ill-posed problems of Hadamard 
\cite{H,L,P}.  One might think that   Theorem \ref{GENERIC} could
be obtained directly  from the 
well-known ill-posedness of the Cauchy problem for the Laplace equation
on a strip.

 However, this is not the case:  the leafwise equations \eqref{LeafwiseDefEq}
are inhomogeneous and depend on the solution of the HCMA itself.
Second, and perhaps more basic, is that the strip on which the problem is posed depends on the solution
of the HCMA. The standard argument of Hadamard 
(for the classical Laplace equation) 
of perturbing the Cauchy data
so as to lie outside the range of the 
Dirichlet-to-Neumann operator therefore cannot be applied directly
as it would also perturb the leaves themselves!

The actual proof  does employ
the Dirichlet-to-Neumann operators along each leaf but also uses a geometric perturbation argument. 
First,   we analyze the obstructions for a leafwise subsolution
in detail, and show that, for each $z \in M$, the pull-back  
of a leafwise subsolution under $\gamma_z$ (the map of restriction
to the leaf through $z$ defined in Definition \ref{OptimalSubsolutionDef}) satisfies
a certain real-analyticity condition on the initial boundary $\{0\}\times\RR$ of $S_T$; 
more precisely,
a certain function 
of the Cauchy data is real
analytic on $\R$ and possesses an analytic continuation to a {\it two-sided} 
strip $[-T,T]\times \RR$  
of width precisely $T$. We refer to Proposition \ref{LeafwiseProp} for the precise statement.
Second, we combine 
Theorem \ref{HCMACauchyCthreeThm} and Proposition \ref{LeafwiseProp}
with a geometric perturbation argument and basic properties
of the Hilbert transform to derive a real-analyticity condition,
independent of $T$.

Theorem \ref{GENERIC} is thus based on the analytic continuation obstruction of 
 Theorem \ref{HCMACauchyCthreeThm}. In the remainder of the paper we concentrate
on the second obstruction, i.e., the  invertibility of the Moser maps 
$f_s$ appearing in Theorem \ref{HCMACauchyCthreeThm}.
It is present even in the simplest case of toric K\"ahler manifolds. 
As we will see, even  when the strip-wise Cauchy problems can all be solved, there does not generally exist a global in time solution
of the HCMA.

\subsection{Complementary results for the HRMA}

The local ill-posedness result, Theorem \ref{GENERIC}, does not apply   to the study of the HRMA. The Cauchy problem
for the HRMA arises precisely when the Cauchy data is torus-invariant,
which is, of course, non-generic in the space of all possible Cauchy data.
And in fact, the Cauchy problem for the HRMA has a positive smooth
lifespan 
\cite{RZ2}.  Moreover, as we observe in Proposition \ref{ToricLifespanProp},
there is no obstruction to analytically continuing orbits. In the remainder of the article
our goal is thus to derive results for the HRMA
that are somewhat of a complementary nature to those for the HCMA described
above. First, we would like to understand how our characterization
of the smooth (or $C^3$) lifespan specializes to the setting of the HRMA. Second,
we would like to understand lifespan of solutions
with less regularity than that described in 
Theorem~\ref{HCMACauchyCthreeThm}, that is less than $C^3$.

Theorem \ref{HCMACauchyCthreeThm} clarifies  the breakdown
of classical solutions of the Cauchy problem for the HCMA already in the toric case.
When the Cauchy data is $(S^1)^n$-invariant, the equation reduces to the HRMA
\begin{equation}
\label{HRMARayEq}
\begin{aligned}
\!\!\h{\rm MA}\, \psi
=
0, \;\,\mskip2mu &\mbox{on} \; [0,T] \times \RR^n,
\quad\;
\psi(0,\,\cdot\,)
=
\psi_0(\,\cdot\,),
\quad\;
\partial_s\psi(0,\,\cdot\,)
=
\dot\psi_0(\,\cdot\,), \;\; \mbox{on} \; \RR^n,
\end{aligned}
\end{equation}
that describes geodesics in the space of toric \K metrics, where $\h{MA}$
denotes the real \MA operator that 
associates a Borel measure to a convex function 
and equals $\det\nabla^2 f\, dx^1\w\cdots\w dx^{n+1}$ on $C^2$ functions
(see \cite[\S2.2]{RZ2} and \S\ref{ConeLifespanSection}).
Here $\psi_0$ is a smooth strictly convex function (moreover,
with strictly positive Hessian) on $\RR^n$ that corresponds to
a torus-invariant \K metric, i.e.,
$\o_{\vp_0}=\i\ddbar\psi_0$ over the 
open orbit and $\dot\vp_0$ is a smooth torus-invariant function on $M$,
considered as a smooth bounded function $\dot\psi_0$ on $\RR^n$. Thus, we view $\RR^n$
as the real slice of $M$ (minus its divisor at infinity). We also refer to the real slices of the leaves
of the \MA foliation as leaves.
Also, $\overline{\Im\nabla\psi_0}=P$ is a compact convex polytope in $\RR^n$.
Translating the definition of $\pi_2^*\omega_0$-psh  solutions to the HCMA \eqref{HCMARayEq}
to the real setting yields a corresponding class for the HRMA. 

Before defining the class we recall the   definition of the \MA operator. 
Let $M(\RR^{n+1})$ denote the space of differential forms of degree $n+1$ on $\RR^{n+1}$
whose coefficients are Borel measures (i.e., currents of degree $n+1$ and order 0).

\begin{prop} {\rm (See \cite[Proposition 3.1]{RT})}
Define by 
$$
\h{\rm MA} f:=
d\frac{\partial f}{\partial x^1}\w \cdots\w d\frac{\partial f}{\partial x^{n+1}},
$$
an operator $\h{\rm MA}: C^2(\RR^{n+1})\ra M(\RR^{n+1})$.
Then $\h{\rm MA}$ has a unique extension to a continuous operator on the cone
of
convex functions.
\end{prop}

A result of Alexandrov shows that 
for any convex function $f$, the measure $\calM\calA\, f$, defined
by
$
(\calM\calA\, f)(E):=\h{\rm Lebesgue measure of\ } \partial f(E),
$
where $\del f$ denotes the subdifferential mapping of $f$ (see \cite[\S2.1]{RZ2}),
is a Borel measure (\cite[Section 2]{RT}).
Furthermore,
according to Rauch--Taylor 
$\h{\rm MA} f = \calM\calA\, f$
for every convex function $f$ on $\RR^{n+1}$ \cite[Proposition 3.4]{RT}.

\begin{definition} 
\label{AlexandrovDef}
A convex function $\rho$ on $[0,T]\times\RR^n$ is an Alexandrov weak solution of the HRMA
$\h{\rm MA}\, \rho =0$, 
if the image of the subdifferential mapping $\del\rho:[0,T]\times\RR^n\ra\RR^{n+1}$
is a set of Lebesgue measure zero. 
\end{definition}

We now define our class of ``admissible solutions" to the HRMA to be Alexandrov weak solutions
with the property that the image under the spatial sub-differential of the solution  is a fixed polytope for all times. 
The assumption  means that the solutions are potentials of  non-degenerate \K metrics for each $s$
that stay in the same \K class.  It seems that  only  these
solutions are relevant to toric  K\"ahler geometry.

\begin{definition}
\label
{AdmissibleDef}
An admissible subsolution to \eqref{HRMARayEq} is a 
convex function on $[0,T]\times\RR^n$ satisfying
(i) $\psi(0,\,\cdot\,) = \psi_0(\,\cdot\,)$ and
$\partial_s\psi(0,\,\cdot\,) = \dot\psi_0(\,\cdot\,)$
on $\RR^n$,
and 
(ii) $\psi(s):\RR^n\ra\RR$ is strictly convex,
and
$\overline{\Im\,\del\psi(s)}=P$ for each $s\in[0,T]$.
An admissible solution in addition is a weak solution
of $\h{\rm MA}\, \psi = 0$ in the sense of Alexandrov.
\end{definition}

Note that this definition assumes $\psi_0$ and $\dot\psi_0$ to satisfy
the regularity, growth, and convexity assumptions of the previous paragraphs. 

In the previous article, we showed that the Legendre transform method
for solving the HRMA breaks down at the {\it convex lifespan } 
\begin{equation}
\label{ToricTspanEq}
T_{\span}^\cvx
:=
\sup\,\{\,s>0: \psi^\star_0-s\dot \psi_0\circ(\nabla\psi_0)^{-1} \hbox{\rm\ is convex}\},
\end{equation}
where $\psi_0^\star$ denotes the Legendre transform of $\psi_0$ \cite[Theorem 1]{RZ2}.
The next proposition shows that there
is no obstruction for the Hamilton orbits to admit 
analytic extensions to strips nor for the maps \eqref{MAPS} 
to be smooth, and that the only
obstruction to smooth solvability is the invertibility of
these maps, that we refer to as Moser maps (see Definition \ref{MoserMapDef}).

\begin{prop}
\label
{ToricLifespanProp} 
Let $(M,J,\o_{\vp_0})$ be a toric \K manifold,
and let $\dot \vp_0\in C^3(M)$ be torus-invariant.
Then,\hfill\break
(i) The Cauchy problem (\ref{HCMARayEq}) for $(\o_{\vp_0},\dot\vp_0)$ is $T$-good for
every $T>0$. \hfill\break
(ii) The  
maps $f_s(z)=\exp-\i sX_{\dot\vp_0}^{\o_{\vp_0}}.z$ (\ref{MAPS})  are invertible
if and only if $\,s\in[0,T^\cvx_\span)$.
\end{prop}

This result, together with Theorem \ref{HCMACauchyCthreeThm},
determines the smooth lifespan for toric geodesics,
as well as characterizes all smooth toric geodesic rays.

\begin{cor}
\label
{ToricGeodesicRayCor} {\rm (Characterization of smooth toric geodesics)}
(i) The smooth lifespan of the Cauchy problem \eqref{HRMARayEq}
with smooth Cauchy data
coincides with the convex lifespan \eqref{ToricTspanEq}, 
$T^\infty_\span=T_\span^\cvx$.
\hfill\break
(ii)
Smooth geodesic rays in the space of toric metrics 
are in one-to-one correspondence with admissible solutions
of the Cauchy problem (\ref{HRMARayEq}) with $\psi_0\in C^\infty(\RR^n)$,
$\nabla^2\psi_0>0$,
$\overline{\Im\nabla\psi_0}=P$,
$\dot\psi_0\in C^\infty\cap L^\infty(\RR^n)$, and $\dot \psi_0\circ(\nabla\psi_0)^{-1}$ a 
concave function on $P$.
\end{cor}

Next, we show that in the case of the HRMA the leafwise
obstruction vanishes and characterizes the Legendre transform subsolution
among all subsolutions of the Cauchy problem.

\begin{prop}
\label{OptimalSubsolutionToricProp}
(i) The Legendre transform potential,  given by
\begin{equation}
\label{OptimalSubsolutionToricEq}
\psi_L(s,x)
:=
(\psi^\star_0-s \dot\vp_0\circ(\nabla\psi_0)^{-1})^\star(z),\quad x\in \RR^n, \;
s\in\RR_+,
\end{equation}
is the unique admissible leafwise subsolution to 
the HRMA (\ref{HRMARayEq}) for all $T>0$.
\hfill\break
(ii)
The corresponding unique admissible leafwise subsolution to the HCMA (\ref{HCMARayEq}) is
given by
\begin{equation}
\label
{ToricHCMALeafwiseSubSolEq}
\vp_L(s+\i t,e^{x+\i\th}):=\psi_L(s,x)-\psi_0(x).
\end{equation}
\end{prop}

Observe that the uniqueness result in (i) holds under much
weaker regularity than that needed in Theorem
\ref{HCMACauchyCthreeThm}.

However, the possibility remains that a solution could persist
beyond $\Tspancvx$,
but not be given by the Legendre transform method.  
But by following the lead of Theorem \ref{HCMACauchyCthreeThm} in the case
of the HRMA, 
we show that there cannot exist any $C^1$ weak solution
in the Alexandrov sense beyond $T^{\cvx}_{\span}$.  
The result is a regularity statement.

\begin{thm}
\label
{ConeLifeSpanThm} {\rm ($C^1$ lifespan of HRMA)}
Any admissible $C^1$ weak solution to the Cauchy problem (\ref{HRMARayEq})
with Cauchy data $\psi_0\in C^\infty(\RR^n)$,
$\nabla^2\psi_0>0$,
$\overline{\Im\nabla\psi_0}=P$,
$\dot\psi_0\in C^\infty\cap L^\infty(\RR^n)$
is smooth.
Thus, $T^1_\span=T^\cvx_\span$.
\end{thm}

This generalizes a classical theorem of Pogorelov on
the developability of flat (in a suitable sense) $C^1$ surfaces in $\RR^3$.
In the language of geodesics in the infinite dimensional symmetric
space $\hcal_\o$ \eqref{HoEq} \cite{M,S,D1}, 
it shows that the exponential map fails to be globally defined even when $C^1$ weak 
solutions are allowed. 
It is interesting to observe that Pogorelov's
result for $n=1$ involves a quite intricate proof \cite{Po,Sa}. In higher dimensions,
this result has been known previously 
under the rather stronger assumption of $C^2$ regularity or more,
i.e., for classical solutions \cite{HN,Fo1,Fo2,U}.

The proof of Theorem \ref{ConeLifeSpanThm} uses the following  characterization
of the HRMA in terms of a Hamilton--Jacobi equation:

\begin{thm}
\label{HJThm} {\rm (HRMA and Hamilton--Jacobi)}
$\eta\in C^1([0,T\times\RR^n)$ is an admissible weak solution of the HRMA \eqref{HRMARayEq}
if and only if it is a classical solution of the Hamilton--Jacobi equation
\begin{equation}
\label{HJEq}
F(\nabla \eta)=0, \qquad \eta(0,\,\cdot\,)=\psi_0,
\end{equation}
where $F(\sigma,\xi)=\sigma-\dot \psi_0\circ (\nabla\psi_0)^{-1}(\xi)$, where
$\sigma\in\RR,\xi\in \RR^n$.
\end{thm}

Theorem \ref{HJThm} reduces the HRMA to a first-order equation
for which a well-known theory for solutions exists---based on the method
of characteristics.  The Hamilton--Jacobi equation is a `conservation law'
for the HRMA. 
It may be viewed as combining the conservation law  
$\dot{\varphi}_s \circ f_s = \dot{\varphi}_0$
of Proposition \ref{FTPHIDOT}  (see \eqref{MAPS} for notation) with the
 explicit formula for $f_s^{-1}$ in    \eqref{SecondAllTimeFsToricMomentEq}; see also \eqref{VariationPotentialEq}.
This makes rigorous as well as generalizes to weak solutions the folklore idea \cite{HN,Fo1,Fo2,U} that classical solutions to HRMA---despite being of second-order---can be obtained by integrating along `characteristics' just like a first-order equation, indeed they are affine along lines determined by the Cauchy data.

Theorem \ref{ConeLifeSpanThm}  follows from Proposition \ref{OptimalSubsolutionToricProp}
and  Theorem \ref{HJThm}.  Except for one step (Proposition \ref{ConeadmissibleIsOptimalProp}), the proof is short and we give it here:

\begin{proof}[Proof of Theorem \ref{ConeLifeSpanThm}]

Given the results of \cite{RZ2}, 
the  main new  step of the proof of Theorem \ref{ConeLifeSpanThm} is the following generalization to 
weak $C^1$ admissible solutions of HRMA~\eqref{HRMARayEq} 
of the fact (see
 Section \ref{HamFlowsHCMASection})  that 
every $C^3$ $\pi_2^\star\o$-psh solution of the HCMA~\eqref{HCMARayEq} is a leafwise subsolution.

\begin{prop}
\label{ConeadmissibleIsOptimalProp}
Let $(M,J,\o_{\vp_0})$ be a toric \K manifold,
and let $\dot \vp_0\in C^\infty(M)$ be torus-invariant.
Assume that the
corresponding Cauchy problem for the HCMA (\ref{HCMARayEq}) is $T$-good.
Then any $C^1$ $\pi_2^\star\o$-psh solution of the HCMA (\ref{HCMARayEq}) up to time $T$ is
the unique $T$-leafwise subsolution.

\end{prop}

The proof of Proposition \ref{ConeadmissibleIsOptimalProp} is based on Theorem \ref{HJThm} 
and uniqueness of $C^1$ solutions of the Hamilton--Jacobi equation. 

We now complete the proof Theorem \ref{ConeLifeSpanThm},
assuming Proposition \ref{ConeadmissibleIsOptimalProp}.   This is possible since the $T$-good assumption is satisfied in the toric setting.   The proof is simple
and is given in  Lemma \ref{AnalyticContinuationHamOrbitsToricLemma}.
By
Proposition \ref{OptimalSubsolutionToricProp}
there exists a unique leafwise subsolution $\vp_L$ of the toric HCMA 
(see \eqref{ToricHCMALeafwiseSubSolEq}), induced by the Legendre transform
potential \eqref{OptimalSubsolutionToricEq}. By Proposition~\ref{ConeadmissibleIsOptimalProp}
 any $\pi_2^\star\o$-psh $C^1$
solution of \eqref{HCMARayEq} on a toric variety must coincide 
with $\vp_L$. However,
 $\vp_L\not\in C^1$
for $T>T_\span^\cvx$
\cite[Proposition~1]{RZ2}.  
Hence, there exists no admissible $C^1$ weak solution of the IVP for $T>T_\span^\infty$, concluding  the proof of
Theorem \ref{ConeLifeSpanThm}. 
\end{proof}

For sufficiently regular $\eta, $ Theorem 1.13 can be proved in a
symplectic geometric way by observing that the Lagrangian submanifold
$\Lambda_{\nu}: = \mbox{graph} (d \eta)$ of $T^*(\R \times \R^n)$ lies in
a level set of the Hamiltonian $F$. When $\Lambda_{\eta}$ is
sufficiently smooth, it must then be invariant under the Hamilton flow
of $F$.  When $\Lambda_{\eta}$ is Lipschitz, for instance, we can use flat forms and
chains to prove the latter statement, and obtain:

\begin{prop} 
\label
{COR}
Let  $\eta\in C^{1,1}([0,T]\times\RR^n)$ be 
an admissible weak solution of the HRMA.
Then the Lipschitz
Lagrangian submanifold $\Lambda_{\eta}:=\h{graph}\,(d\eta)\subset T^\star\RR^{n+1}$ 
is foliated by straight line segments along each of which $\nabla\eta$ is constant.
\end{prop}

We postpone the details of this symplectic approach to the HRMA and the proof of this proposition
to a sequel \cite{RZ3}, where we also pursue a complex analogue for the HCMA.

\subsection{Organization}
The charactrization of the smooth lifespan and uniqueness of classical solutions 
(Theorem~\ref{HCMACauchyCthreeThm}) is proved in Section~\ref{HamFlowsHCMASection}.
The ill-posedness of the leafwise problem, Theorem~\ref{GENERIC}, is proved in 
Section~\ref{LeafwiseSection}. Proposition~\ref{ToricLifespanProp} concerning the obstructions
to solvability and the characterization of the smooth lifespan in the toric setting
is proved in Section~\ref{LifespanHRMASection}. The characterization of the Legendre
potential as the unique leafwise subsolution is proved in 
Section~\ref{OptimalSubsolutionHRMASection}. The characterization of the $C^1$ lifespan
for the HRMA is given in Section~\ref{ConeLifespanSection}, where we also prove the equivalence between
HRMA and a Hamilton--Jacobi equation.

\section{Smooth lifespan of the HCMA: Proof of Theorem \ref{HCMACauchyCthreeThm}}
\label{HamFlowsHCMASection}

Before proving Theorem \ref{HCMACauchyCthreeThm}, we need to introduce some terminology and background  related to the
ill-posedness of the Cauchy problem.

\begin{definition}
\label{THamAnalyticDef} We say that the Cauchy problem (\ref{HCMARayEq})
with smooth initial data $(M, \omega_{\vp_0},\dot\vp_0)$ is 
$T$-Hamiltonian analytic if for every $z \in M$ the orbit
of $z$ under the Hamiltonian flow of $\dot\vp_0$ with respect to
$\omega_{\vp_0}$ admits a holomorphic extension to the strip
$S_T$. 

\end{definition}

Here, by a holomorphic extension of a map $\gamma:\RR\ra M$ to
$S_T$ we mean a holomorphic map $\tilde\gamma:S_T\ra M$ such that
$\tilde\gamma(0,t)=\gamma(t)$. 
Such an extension is unique when it
exists (this can be seen either by the Cauchy-Riemann equations or
by the Monodromy Theorem). When it exists for  $z\in M$ we denote it
by  $\exp-\i\tau X_{\dot\vp_0}^{\omega_{\vp_0}}.z, \,\tau\in
S_T$  the holomorphic strip extending the Hamiltonian orbit
$\exp tX_{\dot\vp_0}^{\omega_{\vp_0}}.z$.

\begin{definition}\label
{MoserMapDef}
The Moser maps are defined by
\begin{equation}
\label{FsMapEq} 
f_\tau(z):= \exp-\i \tau
X_{\dot\vp_0}^{\omega_{\vp_0}}.z, \quad \tau=s+\i t\in S_T.
\end{equation}
\end{definition}

Thus, by definition, the Moser maps are the analytic continuation
to complex time of the Hamiltonian flow of $\dot\vp_0$ with respect
to the symplectic structure $(M,\o_{\vp_0})$. This terminology will be justified
by the fact that for solutions of the HCMA these maps 
act as Moser maps in the usual sense of symplectic geometry, see
\eqref{FtauPullbackEq} below.

\begin{definition} \label{TGoodDef} 
We say that the Cauchy problem (\ref{HCMARayEq}) is $T$-good if it is
$T$-Hamiltonian analytic and if the Moser map $f_{\tau}$ is  a
differentiable map of $M$ for each $\tau\in S_T$.
\end{definition}

\subsection{HCMA and invertibility of the Moser maps}
\label{HCMAInvertibilityMoserSubsection}

We now begin the proof of Theorem  \ref{HCMACauchyCthreeThm}.

In this subsection we show one direction, namely, 
that a $C^3$ solution of the HCMA gives rise to 
smoothly invertible Moser maps in the sense of Definition \ref{MoserMapDef}.
The proof 
can be extracted from the arguments of \cite{S,D1}.
For the sake of completeness, we present the rather simple argument.

A $C^3$ function $\vp$ on $S_T\times M$ satisfies HCMA if and only if the
form  $\pi_2^\star\omega+\i\ddbar\vp$ has a non-trivial kernel.
Since this form is of type $(1,1)$ and, 
according to (\ref{HCMARayEq}), nondegenerate
on $M$-slices it follows that the
kernel defines a one-dimensional integrable complex distribution on $S_T\times M$.
It also follows that its leaves are holomorphic copies of $S_T$
inside $S_T\times M$, and  the leaf passing through $(0,z)$, for each $z\in M$, may be parametrized in the form 
\begin{equation} \label
{GAMMADEF} 
\{(\tau,\Gamma_z(\tau))\,:\, \tau\in S_T\}\subset S_T\times M,\;\;\; \text{with}\;\Gamma_z: S_T \to M. 
\end{equation}
For each $\tau\in S_T$ define the map $f_\tau: M \to M$ by
\begin{equation} \label
{ftau} 
f_\tau(z):=\Gamma_z(\tau).  \end{equation}   
By the transversality condition and the fact that the leaves
do not intersect each other (follows from uniqueness for ODEs with $C^1$
coefficients---here  we used the $C^3$ assumption for the second time) it follows that $f_\tau$
is a $C^1$ diffeomorphism. 

It remains to prove that the maps $f_\tau$ are Moser maps in the sense of
Definition \ref{MoserMapDef}.
Since the strips are constructed
by integrating the 
vector field $\frac{\del}{\del\tau}+\frac{d f_\tau}{d\tau}$
in $S_T\times M$, this vector field lies in the kernel of
$\pi_2^\star\omega+\i\ddbar\vp$. Therefore,
\begin{equation}
\label
{FtauPullbackEq} 
f_\tau^\star\o_{\vp_\tau}=\o_{\vp_0}.
\end{equation}

Now, since $\vp_0,\dot\vp_0$ are invariant under the $\RR$-action
$(\tau,z)\mapsto (\tau+\i c,z), c\in\RR$, uniqueness of smooth
solutions implies that so is $\vp_\tau$. The proof of uniqueness is postponed to Lemma
\ref{ExistenceGivenTGoodInvertibleLemma} below, however its proof 
does not rely on the rest of this subsection. 
By abuse of notation we write $\vp_s=\vp_\tau$ when no confusion arises, where
$\tau=s+\i t$. 

Next,
\begin{equation}\label
{FlowLeavesEqs}
\frac{df_\tau}{dt}=X_{\dot\vp_s}^{\o_{\vp_s}}\circ f_\tau=-J\nabla_{g_{\vp_s}}{\dot\vp_s}\circ f_\tau,\qquad
\frac{df_\tau}{ds}=-\nabla_{g_{\vp_s}}{\dot\vp_s}\circ f_\tau,\quad
f_0=\id,
\end{equation}
since $\frac{\del}{\del\tau}-\nabla^{1,0}_{g_{\vp_s}}{\dot\vp_s}\in\ker
(\pi_2^\star\omega+\i\ddbar\vp)\big|_{(\tau,\Gamma_z(\tau))}$, indeed
$$
\iota_\frac{\del}{\del\tau}(\pi_2^\star\omega+\i\ddbar\vp)=
\i\dbar\frac{\del\vp}{\del \tau}=\i \dbar\dot\vp_s,
$$
and
$$
\iota_{\nabla_{g_{\vp_s}}{\dot\vp_s}}(\pi_2^\star\omega+\i\ddbar\vp)
=
\iota_{\nabla_{g_{\vp_s}}{\dot\vp_s}}\o_{\vp_s}
=
d^c\dot\vp_s=\i(\dbar-\del)\dot\vp_s,
$$
and we use the convention $\frac{\del}{\del \tau}=\frac12\frac{\del}{\del s}-\frac\i2\frac{\del}{\del t}$
and $Y^{1,0}=\frac12Y-\frac\i2JY$.

It then follows from (\ref{FtauPullbackEq}) that
\begin{equation}
\label
{MoserDiffDecompositionEq}
f_{s+\i t} = h_{s+\i t}\circ f_{s},
\end{equation}
with $h_{s+\i t}$ a $C^1$ symplectomorphism of $(M,\omega_{\vp_s})$.
Also, from \eqref{MoserDiffDecompositionEq} and \eqref{FlowLeavesEqs} 
\begin{equation} \label{h} 
h_{s+\i t}=\exp tX_{\dot\vp_s}^{\o_{\vp_s}}.
\end{equation}
We conclude therefore from \eqref{MoserDiffDecompositionEq}
and \eqref{FlowLeavesEqs} that the maps $f_\tau$ defined by \eqref{ftau} satisfy \eqref{FsMapEq},
i.e., for each $z\in M$, induce analytic continuation to the strip of the Hamiltonian
orbit $\exp tX_{\dot\vp_0}^{\o_{\vp_0}}.z$. Hence we have shown both that 
the Cauchy data is $T$-good and that
the Moser maps of Definition \ref{MoserMapDef} are $C^1$ and admit $C^1$ inverses
for each $s\in[0,T]$.
This completes the proof of the first half of Theorem \ref{HCMACauchyCthreeThm}.

We conclude this subsection with some further properties of the Moser maps.
In view of \eqref{FORMAL}, the Moser maps which are relevant to the solution of HCMA  are the ones with
$t = 0$, and their definition only requires analytic continuation of the 
Hamiltonian flow of $X_{\dot\vp_0}^{\o_{\vp_0}}$ to a rectangle 
$[0, T]\times(-\epsilon, \epsilon)$.
However, such an analytic continuation necessarily induces one to the strip $S_T$.

\begin{cor} \label{STRIP}  Suppose that $h_{\sqrt{-1} t}= \exp tX_{\dot\vp_0}^{\omega_{\vp_0}} z$ admits an analytic continuation to
$[0, T]\times(-\epsilon, \epsilon)$.
Then $h_{\sqrt{-1}t }$   admits an analytic continuation to
$S_T$. \end{cor}

Indeed, by \eqref{h},  $h_{s + \sqrt{-1}t}(z)$ is the orbit of a Hamiltonian flow for fixed $s$ and varying 
$t\in \RR$. 
Hence it may be holomorphically extended by the group law 
$$\exp (t_1 + t_2) X_{\dot\vp_s}^{\o_{\vp_s} }(z) = \exp t_1 X_{\dot\vp_s}^{\o_{\vp_s}} (\exp t_2X_{\dot\vp_s}^{\o_{\vp_s}}z). $$
Therefore, by \eqref{MoserDiffDecompositionEq}, one may define $f_{s+\i t}$ for all $s+\i t\in S_T$.

\begin{remark} 
\label{NoGroupLawRemark}
{\rm
In comparison to this group law for fixed $s$,  $f_{s + \sqrt{-1} t}$ does not satisfy a group
law in the complex parameter $s + \sqrt{-1} t$ and thus we cannot conclude that the flow has an analytic continuation
to a half-plane by the same argument. 
This may be seen from the fact that $X_{\dot{\vp}_s}^{\omega_{\vp_s}}$ 
does not Lie-commute with its image under $J$. 
Indeed, commutativity
fails even for generic Cauchy data in the case of toric varieties---see
Remark \ref{ToricNoCommutativityRemark}. 
}
\end{remark}

\subsection{Existence of classical solutions for the HCMA}
\label{AnalyticContinuationHCMASubsection}

In this subsection we continue the proof of Theorem \ref{HCMACauchyCthreeThm},  and
establish the existence of a classical solution to the HCMA under our assumptions.

We now assume that the Cauchy problem for $(\o_{\vp_0},\dot\vp_0)$
is $T$-good and solve the HCMA under the additional assumption of invertibility.

\begin{lem}
\label{ExistenceGivenTGoodInvertibleLemma}
Let $\o_{\vp_0}\in C^1$ and $\dot\vp_0\in C^3$.
Assume that the Cauchy problem for $(\o_{\vp_0},\dot\vp_0)$
is $T$-good, and that for each $\tau\in S_T$ the map $f_\tau$ given by
(\ref{FsMapEq}) is smoothly invertible. Then the HCMA (\ref{HCMARayEq}) admits
a $C^3$ $\pi_2^\star\o$-psh solution.
\end{lem}

\begin{proof}
Define a $C^3$ function on $S_T\times M$ by
\begin{equation}
\label{ExplicitGreenFunctionSolEq} \begin{aligned}
\vp(s+\i t,z)
:=  
&-\i\partial_{\omega_{\vp_0}}^\star \dbar_{\omega_{\vp_0}}^\star
G_{{\omega_{\vp_0}}}^2  \big( (f_\tau^{-1})^\star \omega_{\vp_0}
- \omega_{\vp_0} \big)(z)
\cr
& +\vp_0(z)
+\frac sV\int_M\dot\vp_0\o_{\vp_0}^n, \end{aligned}
\end{equation}
where $G_{{\omega_{\vp_0}}}$ denotes Green's function for the Laplacian
$\Delta_{{\omega_{\vp_0}}}=-\dbar\circ\dbar^\star-\dbar^\star\circ\dbar$ acting on
forms. The operator $\i\partial_{\omega_{\vp_0}}^\star \dbar_{\omega_{\vp_0}}^\star
G_{{\omega_{\vp_0}}}^2$ is a pseudo-differential operator of order $-2$ with
smooth coefficients. By our assumptions it then follows that $\vp$ is $C^3$.

We claim that $\vp$ solves the HCMA (\ref{HCMARayEq}).
First, observe that since $f_{\i t}(z)=\exp tX_{\dot\vp_0}^{\o_{\vp_0}}.z$
is a symplectomorphism the forumula (\ref{ExplicitGreenFunctionSolEq}) implies
that $\vp(\i t,z)=\vp(0,z)=\vp_0(z)$. 
Next,
$$\begin{aligned}
\frac{\del\vp(\i t,z)}{\del s}
= 
&-\i\partial_{\omega_{\vp_0}}^\star \dbar_{\omega_{\vp_0}}^\star
G_{{\omega_{\vp_0}}}^2
\big(
\calL_{-\frac{df_{\i t}}{ds}}\omega_{\vp_0}
\big)(f_{\i t}(z)) 
\cr
&+\frac1V\int_M\dot\vp_0\o_{\vp_0}^n. 
\end{aligned}
$$
Since 
$$
\frac{df_{\i t}}{ds}=JX_{\dot\vp_0}^{\o_{\vp_0}}(f_{\i t}(z))=-\nabla_{g_{\vp_0}}\dot\vp_0(f_{\i t}(z)),
$$
we have 
$$
\calL_{-\frac{df_{\i t}}{ds}}\omega_{\vp_0}=\i\ddbar\dot\vp_0, 
$$ 
and the $\ddbar$-lemma
\cite[p. 149]{GH} implies that $\frac{\del\vp(\i t,z)}{\del s}=\dot\vp_0(z)$.

Finally, applying the $\ddbar$-lemma again implies that \eqref{FtauPullbackEq}  holds
where $\vp_\tau:=\vp(\tau,\,\cdot\,)$, for all $\tau\in S_T$. Since $f_\tau$
is a diffeomorphism and moreover a smooth homotopy to the identity map it follows
that $\o_{\vp_\tau}$ is a \K metric for each $\tau\in S_T$. In particular,
$(\pi_2^\star\o_{\vp_0}+\i\ddbar\vp)^n\ne0$. Differentiating
(\ref{FtauPullbackEq} ) we find that $\frac{\del}{\del \tau}+\frac{df_\tau}{d\tau}$ is
a holomorphic vector field in the kernel of $\pi_2^\star\o_{\vp_0}+\i\ddbar\vp$.
It follows that $(\pi_2^\star\o_{\vp_0}+\i\ddbar\vp)^{n+1}=0$ on $S_T\times M$,
as required. This concludes the proof of existence.
\end{proof}

\subsection{Uniqueness of classical solutions for the HCMA}
\label{UniquenessHCMASubsection}

In this subsection we complete the proof of Theorem \ref{HCMACauchyCthreeThm},  and
establish the uniqueness of classical solutions to the HCMA under our assumptions.

Before giving the proof let us emphasize some of the subtleties involved.

First, the uniqueness we establish is essentially
equivalent to showing that any solution must be $\RR$-invariant
when the Cauchy data is $\RR$-invariant.
A subtle point is that the HCMA is only equivalent to the geodesic
equation under the assumption of $\RR$-invariance, which is implicit
in the arguments of Semmes and Donaldson. In general, the HCMA
is equivalent to the more complicated WZW equation. Thus, the uniqueness
proof cannot a priori use the identities we established in \S\ref{HCMAInvertibilityMoserSubsection}
for $C^3$ $\RR$-invariant solutions. 
 We need to derive these identities in
the proof, and we do so by first establishing short-time uniqueness and then 
extending this to a global statement.

Thus, if we only wanted to prove uniqueness of
$\RR$-invariant solutions, the proof would simplify considerably. Alternatively, 
one could have defined the class of admissible subsolutions to be $\RR$-invariant
$\pi_2^\star \o$-psh functions. It follows from Lemma \ref{UniquenessGivenTGoodInvertibleLemma}
below that such a restriction would be redundant.

Second, the proof does not follow directly from the arguments
of Bedford--Kalka \cite{BK} and Bedford--Burns \cite[Proposition 1.1]{BB},
where uniqueness is proved for a simpler situation, namely for the equation
$(\i\ddbar u)^m=0$ on $\CC^m$. Parts of the proof are local in nature,
essentially the Cauchy--Kowalevskaya theorem on each strip,
and thus adapt to our setting. However, the relative \K potential $\pi_2^\star\o$
makes the situation more complicated since the leafwise equations are now
not the fixed Laplace equation on $S_T$ but rather an inhomogeneous Poisson equation
that varies from strip to strip, and one has to make sure that this equation does not
degenerate. Thus, we need to invoke a global conservation law type argument that
is special for our HCMA \eqref{HCMARayEq}.

\begin{lem}
\label{UniquenessGivenTGoodInvertibleLemma}
Let $\o_{\vp_0}\in C^1$ and $\dot\vp_0\in C^3$.
Assume that the Cauchy problem for $(\o_{\vp_0},\dot\vp_0)$
is $T$-good, and that for each $\tau\in S_T$ the map $f_\tau$ given by
(\ref{FsMapEq}) is smoothly invertible. 
Then any $C^3$ $\pi_2^\star\o$-psh solution of the HCMA (\ref{HCMARayEq}) 
is unique, and in particular $\RR$-invariant.
\end{lem}

\begin{proof}
Assume that $\vp,\rho\in C^3$ are
both $\pi_2^\star\o$-psh solutions of (\ref{HCMARayEq}). 
Then the equation \eqref{HCMARayEq} and the equality of the Cauchy data 
implies that all the second
derivatives of $\vp$ and $\rho$, possibly with the exception of the second $s$ derivative,
agree on the hypersurface $\Sigma:=\{0\}\times\RR\times M$.
Now the form
$\pi_2^\star\omega+\i\ddbar\Phi$ restricts to a positive form on
$\Sigma$ ensuring that $g_\vp$ is non-degenerate
(i.e., $\Sigma$ is non-characteristic). Also, 
the Monge--Amp\`ere
equation on the initial hypersurface can be rewritten as 
\begin{equation}
 \label{SemmesGeodEq}
\ddot \vp|_{s=0}=\frac12|\nabla\dot\vp|_{g_\vp}^2|_{s=0};
\end{equation}
this was shown by Semmes \cite{S} for all $s$, assuming $\vp$ is an $\RR$-invariant solution, 
but holds  
by his argument at $\{s=0\}$ without that assumption since $\del_t\vp,\del^2_t\vp,\del_s\del_t\vp$,
and $\del_t\del_z\vp$ vanish on $\{0\}\times\RR\times M$
as the initial data is $\RR$-invariant. 
Note that \eqref{SemmesGeodEq} expresses the second $s$ derivative of a solution in
terms of the other second derivatives, all restricted to $\Sigma$. 
Since we know $\vp_0$ is a \K potential, it follows that $\vp$ and $\rho$
agree to second order on $\Sigma$. Thus,
$\ker(\pi_2^\star\o+\i\ddbar\vp)|_\Sigma=\ker(\pi_2^\star\o+\i\ddbar\rho)|_\Sigma$ along the
hypresurface. 
Thus, by the uniqueness of solutions of first order ODEs with $C^1$ coefficients,
the leaves of the foliation by strips defined by each of the solutions $\vp,\rho$
must coincide. Thus the maps defined by \eqref{GAMMADEF} and \eqref{ftau} for
$\vp$ and $\rho$ are identical, and we denote them simply by $\Gamma_z(\tau)=f_\tau(z)$.
By the construction of the \MA foliation, on each leaf 
the \K form $\pi_2^\star\o+\i\ddbar\vp$ satisfies
\eqref{LeafwiseDefEq}.
We claim that \eqref{FlowLeavesEqs} always holds for $s=0$. Recall, that
we proved \eqref{FlowLeavesEqs} for all $s\in[0,T]$, but only under the assumption of $\RR$-invariance
of the solution. To prove this claim, note first
$$
\iota_\frac{\del}{\del\tau}(\pi_2^\star\omega+\i\ddbar\vp)\Big|_{s=0}=
\i\dbar\frac{\del\vp}{\del \tau}\Big|_{s=0}=\i \dbar\dot\vp_0, \quad \h{when $s=0$},
$$
since $\dot\vp_0$ is $\RR$-invariant.
Similarly, since $\vp_0$ is $\RR$-invariant,
$$
\iota_{\nabla_{g_{\vp_0}}{\dot\vp_0}}(\pi_2^\star\omega+\i\ddbar\vp)|_{s=0}
=
\iota_{\nabla_{g_{\vp_0}}{\dot\vp_0}}\o_{\vp_0}
=
d^c\dot\vp_0=\i(\dbar-\del)\dot\vp_0.
$$
Thus,
$\frac{\del}{\del\tau}-\nabla^{1,0}_{g_{\vp_0}}{\dot\vp_0}\in\ker
(\pi_2^\star\omega+\i\ddbar\vp)\big|_{(\i t,\Gamma_z(\i t))}$, 
Therefore, since also $\frac{\del}{\del\tau}+\frac{df_\tau}{d\tau}
\in\ker
(\pi_2^\star\omega+\i\ddbar\vp)\big|_{(\i t,\Gamma_z(\i t))}$,
we conclude that
\begin{equation}\label
{FlowLeavesTimeZeroEqs}
\frac{df_\tau}{dt}\Big|_{s=0}=X_{\dot\vp_0}^{\o_{\vp_0}}\circ f_{\i t}=-J\nabla_{g_{\vp_0}}{\dot\vp_0}\circ f_{\i t},\qquad
\frac{df_\tau}{ds}\Big|_{s=0}=-\nabla_{g_{\vp_0}}{\dot\vp_0}\circ f_{\i t},
\end{equation}
as claimed.

Let $\Gamma_z$ be as in 
\eqref{GAMMADEF} and \eqref{FsMapEq} and suppose that $\Gamma_z(S_T)\not=\{z\}$,
i.e., that the leaf passing through $z$ is not trivial.
For each $z\in M$, put $e_z:=\gamma_z^\star \varphi\,$,
$\tilde e_z:=\gamma_z^\star \rho$, and
let $\omega_z:=\gamma_z^\star\pi_2^\star\omega=\Gamma_z^\star\omega$.
First, note that $\omega_z$ is strictly positive $(1,1)$-form on $S_T$.
Indeed, write $\omega_z=\i a_z d\tau\wedge d\bar\tau=2a_z ds\wedge dt$. 
Then by \eqref{FlowLeavesTimeZeroEqs},
\begin{equation}
\label{AFunctionTimeZeroDefEq}
\begin{aligned}
a_z(\i t)
&= -\i
\omega\Big(d\Gamma_z|_{\tau=\i t}\Big(\frac{\del}{\del\tau}\Big), d\Gamma_z|_{\tau=\i t}\Big(\frac{\del}{\del\bar \tau}\Big)\Big)
\Big|_{\Gamma_z(\i t)}
\cr
&=-\i
\omega\Big( \frac{\del f_\tau(z)}{\del\tau}\Big|_{\tau=\i t}, 
\frac{\del f_\tau(z)}{\del\bar \tau}\Big|_{\tau=\i t} \Big)
\Big|_{f_{\i t}(z)} 
\cr
&=-\frac{\i}4
\omega(-\nabla_{g_{\vp_0}}\dot\vp_0+\i J\nabla_{g_{\vp_0}}\dot\vp_0,
-\nabla_{g_{\vp_0}}\dot\vp_0-\i J\nabla_{g_{\vp_0}}\dot\vp_0)|_{f_{\i t}(z)}
\cr
&=
\frac12\omega(\nabla_{g_{\vp_0}}\dot\vp_0, J\nabla_{g_{\vp_0}}\dot\vp_0)|_{f_{\i t}(z)}
=
\frac12|\nabla_{g_{\vp_0}}\dot\vp_0|^2_{g}({f_{\i t}(z)})\ge0.
\end{aligned}
\end{equation}
Since $g_{\vp_0}$ and $g$ are (strictly positive) metrics, $a_z$ vanishes
at some $\i t\in \{0\}\times\RR\subset S_T$ if and only if $d\dot\vp_0(f_{\i t}(z))=d\dot\vp_0(z)=0$
(by \eqref{FlowLeavesTimeZeroEqs} $f_{\i t}=\exp tX_{\dot\vp_0}^{\o_{\vp_0}}$ so in particular 
$f_{\i t}^\star \dot\vp_0=\dot\vp_0$). Thus, if $a_z(\i t)=0$ for some $t$, then $a_z(\i t)$ for all $t\in \RR$.
Now, for fixed $z\in M$ and $t\in\RR$, equation \eqref{FlowLeavesTimeZeroEqs}
is an ODE in $s$ for $f_{s+\i t}(z)$. If $a_z(\i t)=0$, then its initial condition 
is $f_{\i t}(z)=z$ and the initial derivative is zero. Thus, in this case $f_\tau(z)=z$ for all $\tau\in S_T$,
and the leaf through $z$ is trivial, i.e., $\Gamma_z(S_T)=\{z\}$.
Since we assumed at the beginning of this paragraph that the leaf through $z$ was non-trivial, we thus conclude that $a_z|_{s=0}>0$,
and by continuity also $C>a_z|_{s\in[0,3\eps]}>0$, for some $C,\eps>0$.

Denote the Laplacian associated to $\o_z$ by $\Delta_z$.  Then for
each $z$ with a non-trivial leaf,
the leafwise problem \eqref{LeafwiseDefEq} restricted to $S_{2\eps}\times M$ is equivalent to the Cauchy problem,
\begin{equation}
\label{CauchyLaplaceLeaveLocalEq}
\begin{aligned}
1  +  \Delta_z \a_z
& =
0, \quad \h{\ on \ } S_{2\eps},
\cr
\a_z(\i t)
& =  \vp_0(\Gamma_z(\i t)) \quad \h{\ on \ } \{0\}\times\RR,
\cr\dis
\frac{\del\a_z}{\del s}(\i t)
& = 
\dot\vp_0 (\Gamma_z(\i t)
-
d\vp_0(\nabla_{g_{\vp_0}}\dot\vp_0)(\Gamma_z(\i t)), \quad \h{\ on \ } 
\{0\}\times\RR.
\end{aligned}
\end{equation}
The last equation follows from \eqref{FlowLeavesTimeZeroEqs}.
Thus, $e_z$ and $\tilde e_z$ solve \eqref{CauchyLaplaceLeaveLocalEq}.
Hence, since
$\Delta_z = a_z^{-1} \Delta_0$,
 $\zeta_z:=e_z-\tilde e_z$ solves the Cauchy problem
for $\Delta_0\zeta_z=0$ on $S_{2\eps}$ with zero initial data, where
$\Delta_0$ denotes the Euclidean Laplacian on $S_{2\eps}$. 
It is well-known
that bounded solutions to the Cauchy problem on {\it bounded} domains for this classical Euclidean equation are unique (cf., e.g., \cite[p.~19]{L}). However, we could not find a reference 
that treats our particular situation, namely the non-compact strip as in the following Lemma.


\begin{lem}
\label{LaplaceUniquenessLemma}
Let $u\in C^2\cap L^\infty(S_T)$ be a solution of 
$\Delta_0 u=0$ on $S_T$, with $u|_{s=0}=a\in C^2(\RR)$,
and $\del u/\del s|_{s=0}=b\in C^2(\RR)$. Then $u$ 
is unique.
\end{lem}

\begin{proof}
Since the equation is linear it suffices to consider
the case of zero Cauchy data $a=b=0$, and prove any solution must then vanish.
Also, it suffices to consider the case $T=\pi$, since if $u$ is a non-trivial
solution of $\Delta_0u=0$ on $S_T$ with $a=b=0$ then 
$v(s,t):=u(\frac T\pi s,\frac T\pi t)$ solves the same equation on $S_\pi$.

Let $P$ denote the Dirichlet  Poisson kernel of the strip $S_\pi$,
\begin{equation}
\label{PoissonStripEq}
P(s, t) 
= \frac{\sin s}{\cosh t - \cos s}. 
\end{equation}
According to a theorem of Widder \cite[Theorem 4]{W1}, any harmonic function
bounded below on the strip $S_\pi$ can be expressed as 
\begin{equation}
\label{PoissonStEq}
u(s,t)-\inf u=[Ae^t+Be^{-t}]\sin s+\frac1{2\pi}\int_\RR P(s,a-t)d\alpha(a)
+\frac1{2\pi}\int_\RR P(\pi-s,a-t)d\beta(a),
\end{equation}
for some constants $A,B\ge0$, and some (measurable) nondecreasing functions
$\alpha,\beta:\RR\ra\RR$. Moreover, the integrals converge in the interior
of $S_\pi$. 
Evaluating at $s=0$ gives, by the continuity of $u$
$$
-\inf u =\frac1{2\pi}\lim_{s\ra0^+}\int_\RR P(s,a-t)d\alpha(a).
$$
Thus $-\inf u\, dt=d\alpha(t)$. Plugging this back into \eqref{PoissonStEq},
thus
\begin{equation}
\label{PoissonStSecondEq}
u(s,t)=[Ae^t+Be^{-t}]\sin s+\frac1{2\pi}\int_\RR P(\pi-s,a-t)d\beta(a).
\end{equation}
Therefore,
$$
0=\frac{\del u}{\del s}(0,t)=Ae^t+Be^{-t}
+
\frac1{2\pi}\int_\RR\frac{d\beta(a)}{\cosh(a-t)+1}
$$
Since each of the terms is nonnegative they all vanish. Hence, $A=B=0$,
and $\frac1{2\pi}\int_\RR\frac{d\beta(a)}{\cosh(a-t)+1}=0,$
and therefore $d\beta=0$. Plugging back into \eqref{PoissonStEq},
we conclude that $u=0$, as desired.
\end{proof}


It follows that $e_z=\tilde e_z$, whenever $\Gamma_z(S_T)\not=\{z\}$.
On the other hand, if $\Gamma_z(S_T)=\{z\}$ then $\vp(\tau,z)=\rho(\tau,z)$
by using \eqref{FtauPullbackEq}, \eqref{SemmesGeodEq} and 
that $\vp(\i t,z)=\rho(\i t,z)$.
Since the foliation foliates all of $S_{2\eps}\times M$,
it follows that $\vp=\rho$ on that set. Thus, we have short-time uniqueness for
$C^3$ solutions of the HCMA \eqref{HCMARayEq}. 

In particular, it follows that
both $\vp$ and $\rho$ are $\RR$-invariant for $s\in[0,\eps]$. 
Also, \eqref{MoserDiffDecompositionEq}--\eqref{h} hold since again they were
derived assuming only $\RR$-invariance.
Thus, \eqref{SemmesGeodEq} extends to a strip:
\begin{equation}
 \label{SemmesGeodGlobalEq}
\ddot \vp=\frac12|\nabla\dot\vp|_{g_\vp}^2, 
\end{equation}
on $S_{\eps}\times M$.  Consequently \cite{S,D1}, 
\begin{equation}
\label
{ConservationLawHCMAEq}
\dot\vp_s\circ f_s=\dot\vp_0.
\end{equation} 
Indeed, this holds when $s=0$, and differentiating in $s$ and using
\eqref{FlowLeavesEqs}, \eqref{MoserDiffDecompositionEq}, \eqref{h}, 
and \eqref{SemmesGeodGlobalEq} we obtain it must holds for all $s\in[0,\eps]$,
where we used that $\dot\vp_s$ is constant along its Hamilton orbits
(the factor of $1/2$ in \eqref{SemmesGeodGlobalEq} can be traced to
our normalizations and corresponds to switching between the Hermitian
and the Riemannian metrics associated to $\o_{\vp}$, cf. \cite[\S2.1.4.1,\S2.2.3]{R}).
Finally, we can now also apply \eqref{FlowLeavesEqs} which was valid for any $\RR$-invariant solution,
and compute
\begin{equation}
\label{AFunctionDefEq}
\begin{aligned}
a_z(s+\i t)
&= -\i
\omega\Big(d\Gamma_z\Big(\frac{\del}{\del\tau}\Big), d\Gamma_z\Big(\frac{\del}{\del\bar \tau}\Big)\Big)
\Big|_{\Gamma_z(\tau)}
\cr
&=-\i
\omega\Big( \frac{\del f_\tau(z)}{\del\tau}, \frac{\del f_\tau(z)}{\del\bar \tau} \Big)
\Big|_{f_\tau(z)} 
\cr
&=-
\frac\i4\omega(-\nabla_{g_{\vp_s}}\dot\vp_s+\!\i J\nabla_{g_{\vp_s}}\dot\vp_s,
-\nabla_{g_{\vp_s}}\dot\vp_s-\!\i J\nabla_{g_{\vp_s}}\dot\vp_s)|_{f_{\tau}(z)}
\cr
&=
\frac12\omega(\nabla_{g_{\vp_s}}\dot\vp_s, J\nabla_{g_{\vp_s}}\dot\vp_s)|_{f_{\tau}(z)}
=
\frac12|\nabla_{g_{\vp_s}}\dot\vp_s|^2_{g}({f_{\tau}(z)}).
\end{aligned}
\end{equation}
Therefore, by \eqref{FtauPullbackEq}, \eqref{ConservationLawHCMAEq}, and compactness it follows
that if $0<a|_{s=0}<$ then there exists constants $c,C>0$ determined
by $z$ and the Cauchy data such that the a priori estimate $c<a_z(\tau)<C$ holds for each $\tau\in S_T$
for which a solution exists.
Thus, we can now repeat the argument for the Cauchy
problem with $\RR$-invariant initial data given by $\vp|_{\{\eps\}\times\RR}=\rho|_{\{\eps\}\times\RR}$
and $\dot\vp|_{\{\eps\}\times\RR}=\dot\rho|_{\{\eps\}\times\RR}$, and conclude that in fact
\eqref{CauchyLaplaceLeaveLocalEq} must hold on $S_T$. Thus $\rho=\vp$.
This concludes the proof of Lemma \ref{UniquenessGivenTGoodInvertibleLemma}.
\end{proof}
 
Theorem \ref{HCMACauchyCthreeThm} now follows by combining 
\S\ref{HCMAInvertibilityMoserSubsection},
and Lemmas \ref{ExistenceGivenTGoodInvertibleLemma}
and \ref{UniquenessGivenTGoodInvertibleLemma}.

As can be seen from the proof, $\vp$ is a smooth solution of the IVP \eqref{HCMARayEq}
if and only the Moser maps $f_s$ are smoothly invertible and 
the `conservation law' \eqref{ConservationLawHCMAEq} holds. Of course, this is
a weaker statement than Lemma \ref{ExistenceGivenTGoodInvertibleLemma}. Nevertheless
we record it here.

\begin{cor}
Let $\o_{\vp_0}\in C^1$ and $\dot\vp_0\in C^3$.
Assume that the Cauchy problem for $(\o_{\vp_0},\dot\vp_0)$
is $T$-good, and that for each $\tau\in S_T$ the map $f_\tau$ given by
(\ref{FsMapEq}) is smoothly invertible. Then 
\eqref{SemmesGeodGlobalEq} and \eqref{ConservationLawHCMAEq}
are equivalent.

\end{cor}

We already saw that \eqref{SemmesGeodGlobalEq} implies \eqref{ConservationLawHCMAEq}.
For the converse, note that under the assumptions, it follows from
Lemma \ref{ExistenceGivenTGoodInvertibleLemma} that there exists a solution,
and that the Moser maps determined by the Cauchy data satisfy
\eqref{FlowLeavesEqs}--\eqref{h}; thus differentiating \eqref{ConservationLawHCMAEq}
immediately gives \eqref{SemmesGeodGlobalEq}.

In the setting of the HRMA, we will interpret \eqref{ConservationLawHCMAEq} in terms of a 
Hamilton--Jacobi equation (Theorem \ref{HJThm})
and show that this `conservation law' persists also
for certain weak solutions (Proposition \ref{FTPHIDOT}).

\section{Ill-posedness of leafwise Cauchy problems }
\label
{LeafwiseSection} 

The goal of this section is to prove Theorem \ref{GENERIC}, showing
that the Cauchy problem for the HCMA is not even locally well-posed.
As the proof of Theorem \ref{HCMACauchyCthreeThm} shows, the leaves of the Monge--Amp\`ere foliation 
are obtained as the analytic continuation of the Hamiltonian flow
of $(\omega_{\vp_0},\dot\vp_0)$. The Monge--Amp\`ere distribution
picks out as the $M$-component the Hamiltonian vector field associated to 
$(\omega_{\vp_0},\dot\vp_0)$ and not an arbitrary multiple of it precisely 
because the $S_T$-component of the distribution is $\partial/\partial\tau$.
In other words, the leaves (strips) of the foliation are graphs (of
maps $S_T\ra S_T\times M$) of (complex) time-parametrized Hamiltonian 
flow of $(\omega_{\vp_0},\dot\vp_0)$. 
As we will show, 
this puts a serious restriction on the Cauchy data.

So far, we have operated under the assumption that we have  
$T$-good Cauchy data (Definitions \ref{THamAnalyticDef} and \ref{TGoodDef}). 
Yet the analytic continuation of  each  Hamiltonian orbit should be an 
ill-posed problem. 
The closely related problem of solving the leafwise  Cauchy problem for the
equation \eqref{LeafwiseDefEq} should also ill-posed, and the goal
of this section is to give a proof of this latter ill-posedness.
The latter problem seems simpler than the 
former since it is a linear problem for a function on a strip rather than
a Cauchy problem for a holomorphic map into a nonlinear space.  
Hence we concentrate on the leafwise problem here. However, it is
natural also to linearize the nonlinear problem (cf. \cite{D2}) 
and prove ill-posedness for the existence of $T$-Hamiltonian analytic
data. We pursue this approach in a sequel.

As above, we  suppose that we are given $(\omega_{\vp_0}, \dot{\vp}_0)$  
for which  the orbit $\exp t X_{\dot{\vp}_0}^{\omega_{\vp_0}}$ admits an analytic 
continuation to the strip $S_T$.
Let $\gamma_z$ be as in  Definition \ref{OptimalSubsolutionDef}.
Then  $\a_z = \gamma_z^\star \varphi$ satisfies \eqref{CauchyLaplaceLeaveLocalEq}.
Since $\Gamma_z^\star \omega$ has a global potential $\Phi_z$
on $S_T$, we may also write the equation in terms of the Euclidean Laplacian $\Delta_0$  as
\begin{equation} \label
{DELTAZERO} 
\Delta_0  \chi  = 0, \;\;\; \text{\rm where}\;\;\chi = (\Phi_z + \a_z). 
\end{equation} 
However,
$\Phi_z$ is not unique since the addition of any harmonic function on the strip gives another potential. 
 In the case where the image of the complex Hamiltonian orbit $\Gamma_z$ \eqref{GAMMADEF}
lies in an open set $U \subset M$ in which $\omega$ has a potential $\Phi_0$, we have
$\Phi_z = \Gamma_z^\star \Phi_0$. 
In general, the closure of the image of $\Gamma_z(\i t)$ lies in the level set 
$\{\dot{\vp}_0 = \dot{\vp}_0(z)\}.$
 We will see that toric varieties always satisfy these conditions. 
However, simple examples (e.g., elliptic curves) show that there need not exist
a potential for $\omega$ defined in a neighborhood of the orbit. 
The following lemma shows that one may find a reasonable replacement for
that, on each leaf separately. 
The growth estimate we derive here is not optimal, but suffices for our purposes.

\begin{lem}
\label
{PolyGrowthLemma}  
Let $\vp$ be a smooth solution to the HCMA \eqref{HCMARayEq}.
There exists a global \K potential $\Phi_z$ for $\Gamma_z^\star\o$ on $S_T$
with polynomial growth at infinity.
\end{lem}

\begin{proof}  
The claim would
be obvious if there exists a potential for $\omega$ on a neighborhood of the image of $\Gamma_z(S_T)$,
but as mentioned earlier such a potential need not exist. 
Instead, we will find suitable \K potentials along each leaf.

As before, denote
$\o_z=\gamma_z^\star \pi_2^\star \omega = 2a_z ds \wedge dt.$
As shown in the proof of Lemma \ref{ExistenceGivenTGoodInvertibleLemma}, $a_z>0$
if and only if the leaf through $z$ is non-trivial, i.e., $\Gamma_z(S_T)\not=\{z\}$, which
we assume throughout this section. 
Thus, by compactness, there exist
some constants $c,C>0$ (depending on $z$) 
such that $0 <c < a_z(s, t) < C$ on $S_T$.
For convenience, in this section we omit the subscript and denote $a\equiv a_z$.

 We wish
to find $\Phi_z\in C^\infty(S_T)$ of polynomial growth so that $\Delta_0 \Phi_z = a$,
i.e., $\i\ddbar\Phi_z = \gamma_z^\star \pi_2^\star \omega$. Throughout this proof 
$\del=\del_\tau$.

We rewrite the Poisson equation above as
\begin{equation} \label{PE} \dbar (\i \partial \Phi_z) = -2a ds \wedge dt \end{equation}
and  use existence theorems for the inhomogeneous $\dbar$-equation on the strip.  
Introduce the subharmonic weight $\psi = \log (1 + |\tau|^2)$ and observe that 
$$
a ds \wedge dt \in L^2_{(1,1)}(S_T, \psi ) 
$$
where $L^2_{(1,1)}(S_T, \psi)$ is the space of $(1,1)$ forms $a(s, t) ds \wedge dt$  so that
$$\int_{S_T} e^{- \psi}  |a|^2 ds \wedge dt < \infty. $$
By H\"ormander's weighted $L^2$ existence theorem  for the $\dbar$-equation 
\cite[Theorem 4.4.2]{Ho1}, there exists $u \in L^2_{(0, 1)}(S_T)$ such that
$\dbar u = -2a ds \wedge dt$ and 
$$\int_{S_T} |u|^2 (1 + |\tau|^2)^{-3} d s \wedge dt \leq 4\int_{S_T} |a|^2 e^{- \psi} ds \wedge dt. $$

Applying the same theorem to $\dbar \Phi_z= \bar{u}$
with $\psi = 3 \log (1 + |\tau|^2)$, we
then obtain a solution $\Phi_z$ of $\i \ddbar \Phi = 2a ds \wedge dt$ satisfying
\begin{equation} \label{5} 
\int_{S_T} |\Phi_z|^2 (1 + |\tau|^2)^{-5} ds dt < \infty. 
\end{equation}

We now show that this $L^2$ estimate implies the polynomial growth of $\Phi_z$.
Note that $\partial_s \Phi_z$  and  $\partial_t \Phi_z$  
satisfy a Poisson equation on $S_T$ satisfying the same estimates. 
Indeed, by \eqref{FlowLeavesEqs} (and the assumption of existence of
a smooth solution) under $\Gamma_z$ these 
vector fields
push-forward to the Hamilton vector fields for $\dot{\vp}_s$, respectively $J$ of these fields. 
Hence the Lie derivative
with respect to these fields of $\omega$ are bounded and we can use them as the right hand side in place
of $a ds \wedge dt$ above and repeat the argument to get the estimate \eqref{5} for
these derivatives and  for repeated mixed derivatives. 

By the  Sobolev inequality  
$\sup_{S_T} f^2 \leq C \int_{S_T}  |(1 - \Delta_0) f|^2 ds dt $ for a strip,
we have 
$$
\sup_{S_T} \Phi_z^2 (1 + |\tau|^2)^{-5} 
\leq 
C \int_{S_T} |(1 - \Delta_0)(\Phi_z (1 + |\tau|^2)^{-5/2}) |^2 ds \wedge dt. 
$$
It is straightforward to check that the integral is finite:
this follows from the weighted $L^2$ estimates for $\Phi_z$
and $\Delta_0\Phi_z$, and the fact that derivatives 
of $(1 + |\tau|^2)^{-r} $ 
for $r > 0 $ decay more rapidly with each derivative. 
It follows that 
\begin{equation} \label{EST}
|\Phi_z| \leq C  (1 + |\tau|^2)^{5/2}  \;\; \text{\rm on}\;\; S_T. 
\end{equation}
\end{proof}

\begin{remark}
\rm
In the proof of Theorem \ref{GENERIC} we will be able to specialize to
a situation where $\Phi_z$ is actually of the form $\Gamma_z^\star\Phi_0$.
However, Lemma \ref{PolyGrowthLemma} is needed to derive the general 
obstruction in Proposition \ref{LeafwiseProp} below that holds for all
$z\in M$.
\end{remark}

The obstruction to solvability of \eqref{CauchyLaplaceLeaveLocalEq},
and hence to the existence of a leafwise subsolution (and in particular
to the existence of a $C^3$ solution of the HCMA), is summarized
in the following propostion.

\begin{prop} \label
{LeafwiseProp}  
Let $\vp$ be a 
$C^3$ solution to the HCMA \eqref{HCMARayEq}, and $z\in M$.
Let $D = \frac{1}{\i} \frac{d}{dt}$ on $\R$, and set
$$
\begin{aligned}
q_z(t) 
&: = 
\frac{\partial \a_z}{\partial s}(\i t)
=
\dot\vp_0 (\Gamma_z(\i t)
-
d\vp_0(\nabla_{g_{\vp_0}}\dot\vp_0)(\Gamma_z(\i t)),
\cr
p_z(t)
&: = 
-\del_s\Phi_z(0,t)
-
D\coth TD (\Phi_z+\gamma_z^\star\vp_0)(0,t),
\end{aligned}
$$
where $\Phi_z$ is given by Lemma \ref{PolyGrowthLemma}.
Then,
$$
\widehat{(q_z-p_z)}(\xi)=o(e^{-T|\xi|}).
$$
Thus, $q_z-p_z$ admits
an analytic continuation to the interior of $S_T\cup \overline{S_T}
=[-T,T]\times\RR$.

\end{prop}

Henceforth we denote by $PW_T(\R)$ the Paley--Wiener space
\begin{equation} \label{PW} 
PW_T(\R):= 
\{f \in L^2(\R)\,: \, |\hat{f}(\xi) | =o(e^{-T|\xi|}) \}. 
\end{equation}

Our convention for the Fourier transform is
$$
\calF(f)(\xi)\equiv\hat f(\xi):=\int_\RR e^{-\i t\xi}f(t)dt.
$$
It is well-known  that if $f \in PW_T(\R)$, then $f$ is the
restriction to $\R$ of a holomorphic function on any two-sided strip 
$S_b\cup\overline{S_b}=[-b,b]\times\RR$ with $b < T$   
\cite[p. 121]{St}.

\begin{proof}
Despite the lack of uniqueness of $\Phi_z$ it seems simpler to work with the equation
\eqref{DELTAZERO} rather than $\Delta_z \a = -1$ since the  Euclidean equation 
is simpler and it too has real analytic coefficients.  We then wish to represent the solution 
$\chi_z$ as a 
Poisson integral in terms of
its boundary values on $\partial S_{T}$.  
We first assume $T = \pi$. We recall the following theorem of Widder \cite[Theorem 3]{W1}: If 
$u(s, t)$  is
(i) continuous on $S_{\pi}$ and harmonic on its interior;
(ii) satisfies the bounds $u(0, t) e^{- |t|} \in L^1(\R),\, u(\pi, t) e^{- |t|} \in L^1(\R)$
and $\int_0^{\pi} |u(s, t)| ds = o(e^{|t|}), $
then 
\begin{equation} \label{u} u(s, t) 
= 
\frac{1}{2 \pi} 
\int_{\R} P(s, a - t) u(0,a) da 
+ 
\frac{1}{2 \pi} \int_{\R} P(\pi - s, a- t) u(\pi,a)da, 
\end{equation}
where $P$ is defined by \eqref{PoissonStripEq}.

The assumptions for Widder's theorem are satisfied when 
$$
u=\chi=\a_z+\Phi_z,
$$
with $\Phi_z$ the potential constructed in Lemma \ref{PolyGrowthLemma}. Indeed, then $\chi$
has polynomial growth at infinity on $S_T$
($\a_z$ itself is a bounded continuous function).
Consequently, \eqref{u} is valid when $u = \chi$.

We next consider the implications of this equation for $q_z$. 
As  in \cite{W2}, it simplifies the notation to put
$$
Q(s, t) 
= 
\frac{1}{4} 
\frac
{\cos \frac{\pi}{2} s}
{\cosh \frac{\pi}{2} t + \sin \frac{\pi s}{2}} 
= 
\frac14 P\Big(\frac{\pi}{2} s + \frac{\pi}{2}, \frac{\pi}{2}  t\Big)
$$
on the strip $s \in (-1,1), t \in \R$. 
One has \cite[(5)]{W2} 
$$
Q(s, t) 
= 
\frac{1}{2\pi} \int_{\R} 
e^{-\i t a}\;\; \frac{\sinh(1-s)a}{\sinh 2a} da, \;\; s \in (-1, 1). 
$$
Then, 
\begin{equation}
\label{UstEq}
\begin{aligned}
u(s, t) 
& =  
\int_{\R} Q(s,a-t) u(-1,a)da + \int_{\R} Q(-s,a-t) u(1,a)da
\cr
& =  
\int_{\R} e^{\i t a} \frac{\sinh (1 - s) a}{\sinh 2 a} \hat u(-1,a)da 
+ 
\int_{\R} e^{\i t a} \frac{\sinh (1 + s) a}{\sinh 2 a} \hat u(1,a)da.
\end{aligned}
\end{equation}
Note that this formula holds even when $u(\pm 1,\,\cdot\,)$ is of polynomial
growth. Then $\hat u(\pm 1,\,\cdot\,)$ is a temperate distribution
while
$\frac{\sinh (1 \mp s) a}{\sinh 2 a}$ is a 
Schwartz function for $s\in(-1,1)$,
and so the second equality holds by the definition of the Fourier transform
of a temperate distribution \cite[Definition 7.1.9]{Ho2}.

By a change of variable, for the strip $(s,t)\in[0,T]\times\RR$  and for $\Delta_0\chi=0$
with boundary values $\chi(0,\,\cdot\,)$ and $\chi(T,\,\cdot\,)$
we obtain,
$$
\chi(s, t) 
=  
\int_{\R} e^{\i t a} \frac{\sinh (T - s)a}{\sinh T a} \hat \chi(0,a)da 
+ 
\int_{\R} e^{\i t a} \frac{\sinh sa}{\sinh T a} \hat \chi(T,a)da.
$$
Thus,
$$
\begin{aligned} 
\del_s\chi(0,t)
= 
-
\int_{\R} e^{\i t a} a\coth Ta \hat \chi(0,a) da 
+ 
\int_{\R} e^{\i t a} \frac{ a}{\sinh T a} \hat \chi(T,a) da.
\end{aligned}
$$
Note that differentiation at the boundary is allowed since we can consider \eqref{UstEq}
as a distributional equation in $t$ with parameter $s$, and so one can pair \eqref{UstEq}
with any Schwartz function of $t$ and then differentiate in $s$.
Thus,
$$
\begin{aligned}
q_z(t)
& =
-\del_s\Phi_z(0,t)
-
D\coth TD \chi(0,\,\cdot\,)
+ 
\frac{D}{\sinh TD}\chi(T,\,\cdot\,)
\cr
& =:p_z(t) 
+ 
\frac{D}{\sinh TD}\chi(T,\,\cdot\,),
\end{aligned}
$$
where $D = \frac{1}{\i} \frac{d}{dt}$ on $\R$.
Inverting, $q_z$ must lie in the domain of the operator 
$\calA_{T,z}:\calS'\ra \calS'$ given
by
$$
\calA_{T,z}:u\mapsto \frac{\sinh TD}{D}(u-p_z)
=
\int_{\RR}e^{ita}\frac{\sinh Ta}{a}\widehat{(u-p_z)}(a)da
.
$$
Here $\calS'$ denotes temperate distributions on $\RR$.
By our earlier estimates $\chi(T,\,\cdot\,)$ is continuous and of at most
polynomial growth, hence belongs to $\calS'$.
Thus it follows that one has 
$\widehat{(q_z-p_z)}(\xi)=o(e^{-T|\xi|})$, since these are the Fourier
coefficients of $\chi(T,\,\cdot\,)$.

In particular $q_z-p_z\in L^2(\RR)$, so $q_z-p_z\in PW_T(\RR)$.
By a Paley--Wiener type theorem  
\cite[p. 121]{St}   
it follows that $q_z-p_z$ admits
an analytic continuation
\begin{equation} \label
{LCALPHI}  
(q_z-p_z) (s + \i t):= \int_{\R} e^{(s +\i t) a} \frac{ a}{\sinh T a} \hat \chi(T,a) da \end{equation} to the interior of a two-sided strip of width $2T$. 
\end{proof}

As mentioned in the Introduction, the Paley--Wiener condition on $q_z$ may be viewed 
as characterizing the range of a 
Dirichlet-to-Neumann map.
These leafwise Dirichlet-to-Neumann maps are induced by the global
Dirichlet-to-Neumann map for the HCMA, defined by
\begin{equation}  
\ncal^T(\vp_0, \vp_T) = \dot{\vp}_0 
\end{equation}
 from the endpoint $\vp_T$ at time $T$ of the geodesic arc from $\vp_0$ to $\vp_T$
 to the initial velocity $\dot{\vp}_0$ of the geodesic.

\subsection{Lifespan of generic Cauchy data}

We now complete the proof of Theorem~\ref{GENERIC}. 

Assume that the Cauchy problem for \eqref{HCMARayEq} with Cauchy data $(\vp_0,\dot\vp_0)$
admits a $C^3$ solution $\vp$.
For simplicity, we take the reference \K metric to be $\omega_{\vp_0}$ and then
the initial relative \K potential becomes zero. 
Then  \eqref{CauchyLaplaceLeaveLocalEq} reduces to

\begin{equation}
\label{CauchyLaplaceLeaveLocalEqa}
\begin{aligned}
1  +  \Delta_z \a_z
& =
0, \quad \h{\ on \ } S_T,
\cr
\a_z(\i t)
& = 0  \quad \h{\ on \ } \{0\}\times\RR,
\cr\dis
\frac{\del\a_z}{\del s}(\i t)
& = 
\dot\vp_0 (z) \quad \h{\ on \ } 
\h{\ on \ } 
\{0\}\times\RR.
\end{aligned}
\end{equation}
The last line follows since $\dot \vp_0$ is constant along its Hamiltonian flow 
orbits. It follows that
$q_z$ is a constant, and therefore Proposition \ref{LeafwiseProp}   implies that
\begin{equation}
\label{ZerothPWConditionEq} p_z = 
(- \partial_s - A_T)  \Phi_z |_{s = 0}\in PW_T(\RR)
\end{equation}
where 
$$
 A_T: =  D \coth T D. 
$$ 
Note that $A_T$ is an approximation (with respect to $T$)  to the 
Dirichlet-to-Neumann operator for the half-plane, which is the operator 
$$|D| f(t) = \int_{\R} e^{\i t a} |a| \hat{f}(a) da. $$

We further observe that, at least for some $z$, $\Phi_z$ is the 
pullback of a \kahler potential defined in a neighborhood
of $\Gamma_z(S_T)$.
Indeed, let $z_0$ be a non-degenerate maximum
point of $\dot{\vp}_0$ (one always exists for a generic
$\dot\vp_0$, which as far as proving Theorem \ref{GENERIC} we may assume is the case), 
so that the orbit of $z_0$ is $\{z_0\}$ 
and find a potential in a neighborhood of $z_0$.   If $z$ is sufficiently 
close to $z_0$ then  the orbit of $z$ under the Hamilton flow of $\dot{\vp}_0$
is non-trivial and is contained in 
the level set  $\{\dot{\vp}_0 = \dot{\vp}_0(z)\}\subset M$, which 
is close to $\{z_0\}$ (by the Morse theorem).  Moreover, this Hamilton orbit
is contractible in $M$ to $z_0$. By \eqref{FlowLeavesEqs}--\eqref{h} the
slices $\Gamma_z(\{s\}\times\RR)$ are all homotopic to the this initial Hamilton orbit.
Hence, $\Gamma_z(S_T)$ itself is contractible to $z_0$. Since $\o$ has a
local \K potential on any contractible neighborhood of $z_0$ in $M$,
the conclusion follows.

Assume from now on that
\begin{equation}\label
{GlobalPotentialAssumpEq}
\Phi_z = \Gamma_z^\star \Phi_0,
\end{equation}
where $\Phi_0$ is a smooth function defined on some
neighborhood of 
$\Gamma_z(S_T)$
in $M$.
For simplicity of notation, put
$$
H:= \dot{\vp}_0, \quad X_H:=X_{\dot\vp_0}^{\o_{\vp_0}}.
$$ 
Let $\nabla$ denote the gradient with respect to the associated metric $g_{\vp_0}$. Then
$$
\partial_s \Phi_z |_{s = 0} = 
\Gamma_z^\star J X_H \Phi_0|_{s=0}. 
$$
By \eqref{FlowLeavesEqs} and \eqref{GlobalPotentialAssumpEq}, 
the conclusion of Proposition \ref{LeafwiseProp}   can be rewritten as
\begin{equation} \label
{SecondPWConditionEq} 
\Gamma_z^\star  J X_H \Phi_0 + A_T \Gamma_z^\star  \Phi_0
=
\Gamma_z^\star d\Phi_0(\nabla H) + A_T \Gamma_z^\star  \Phi_0
 \in PW_T(\RR).  
\end{equation}

We now study this equation under particular deformations
of the Cauchy data.
We denote by 
$$
\ccal_{T,z} = \{(\vp_0, \dot\vp_0,z) \in C^3(M) \times C^3(M)\times M:
(\vp_0,\dot\vp_0) \h{\ is $T$-good and\ }
\eqref{ZerothPWConditionEq} \h{\ holds}\}.
$$
We claim that for  $z$ near a maximum point (as above),  the complement of $\ccal_{T, z}$
in $C^3(M) \times C^3(M)\times M$ is dense. 
By assumption $ (0, \dot{\vp}_0,z)\in\calC_{T,z} $. 
We fix such a $z$ for the rest of the argument. We may, and do, choose $z$
so that in addition it is a regular point for $H$.
Our first goal is to 
find a perturbation of $\dot\vp_0$ with the property that the orbit $\Gamma_z(\i \RR)$ is unchanged. 

Let $h$ be a $C^3$ 
function on $M$, and set for each $\eps\ge 0$,
$$ 
\begin{aligned}
H_{\epsilon} &:= H + \epsilon (H - H(z)) h,
\cr
V_z &:= \{w\in M\,:\, H_\eps(w)=  H_\eps(z))\}.
\end{aligned}
$$
Note that $V_z$ indeed is independent of $\eps\ge0$; by assumption $z\in M$
is a regular point for $H$ so that $V_z$ is a (real) hypersurface.
Also,  
$$
\begin{aligned}
X_{H_{\epsilon} }
:=
X^{\omega_{\vp_0}}_{H_{\epsilon} }
= (1 + \epsilon h) X_H \quad \h{along $V_z$}.
\end{aligned}
$$

Denote $\hat\Gamma_z(t):=\Gamma_z(\i t)$.
Then define
$$
\hat{\Gamma}_z^{\epsilon} ( t):= 
\exp t X_{H_{\epsilon}} (z)
=\hat{\Gamma}_z(  g_\eps(t)),
$$
since $\Gamma_z(\i \RR)\subset V_z$,
where $g_{\epsilon}: \R \to \R$ is a diffeomorphism defined by 
\begin{equation}
\label{gepsprimeEq}
g_\eps'(t):=
\frac{d}{dt} g_{\epsilon}(t) = 1 + \epsilon h\circ \hat{\Gamma}_z( t), 
\quad g_\eps(0)=0. 
\end{equation}
Thus,
\begin{equation}
\label{gepsEq}
g_{\epsilon}(t) = t + \eps\int_0^t h\circ\hat\Gamma_z(a) da. 
\end{equation}

To derive a contradiction, we assume that there exists some $\eps_0>0$ for which
$\{(0,H_\eps,z)\}_{\eps\in[0,\eps_0]}\subset\calC_{T,z}$.
In particular,
by \eqref{SecondPWConditionEq}, for sufficiently small $\eps\ge0$,
\begin{equation}  
\label{EP} 
(\hat\Gamma_z^{\epsilon})^\star d \Phi_0(\nabla H_{\epsilon}) 
+ A_T  (\hat\Gamma_z^{\epsilon})^\star  \Phi_0 \in PW_T(\R). 
\end{equation}
Thus,
\begin{equation}  
\label{EP'}  
\frac{d}{d \epsilon}\Big|_{\epsilon = 0} 
\Big(
(\hat\Gamma_z^{\epsilon})^\star d \Phi_0(\nabla H_{\epsilon}) 
+ A_T  (\hat\Gamma_z^{\epsilon})^\star  \Phi_0 
\Big)
\in PW_T(\R). 
\end{equation}
Note that 
$$ 
d \Phi_0(\nabla H_{\epsilon})|_{V_z}  
= 
(1+\eps h)
d\Phi_0 (\nabla H)|_{V_z}.
$$
Also
\begin{equation}
\label
{EpsVariationAzEq}  
\frac{d}{d \epsilon}\Big|_{\epsilon = 0} 
 (\hat\Gamma_z^{\epsilon})^\star  \Phi_0  = h_z \frac{d}{dt} \hat\Gamma_z^\star \Phi_0,
\end{equation}
where by \eqref{gepsEq}
\begin{equation}
\label{hzEq}
h_z(t)= \int_0^t h\circ\hat \Gamma_z(a) da. 
\end{equation}
Set also,
$$
\tilde h_z=\hat\Gamma_z^\star h.
$$
Then
\begin{equation}
\label{hztildeEq}
h_z'=\tilde h_z.
\end{equation}

First, 
\begin{equation} \label{FIRST}
\begin{aligned}
\frac{d}{d \epsilon}\Big|_{\epsilon = 0} 
(\hat\Gamma_z^{\epsilon})^\star d \Phi_0(\nabla H_{\epsilon}) 
 =
h_z  \big(\hat\Gamma_z^\star d\Phi_0 (\nabla H))\big)' +  
\tilde h_z
\hat\Gamma_z^\star d\Phi_0 (\nabla H). 
\end{aligned}
\end{equation}
Second,
\begin{equation}  \label{EP'2}  
\frac{d}{d \epsilon}\Big|_{\epsilon = 0} 
 A_T (\hat\Gamma_z^{\epsilon})^\star  \Phi_0  = A_T (h_z \frac{d}{dt} \hat\Gamma_z^\star \Phi_0).  
\end{equation}
Combining \eqref{EP'},  \eqref{FIRST} and \eqref{EP'2}, we get that
\begin{equation} \label{VAR}  
h_z  \big(\hat\Gamma_z^\star d\Phi_0 (\nabla H)) \big)'
+  
\tilde h_z \hat\Gamma_z^\star d\Phi_0 (\nabla H)
+ 
A_T (h_z \frac{d}{dt} \hat\Gamma_z^\star  \Phi_0) \in PW_T(\R),
\end{equation}
for all nonnegative functions $h \in C^{\infty}(M)$ for which
$(0,H_\eps,z)\in\calC_{T,z}$.

In this formula,  
\begin{equation}
\label{azbzEq}
a_z:= \hat\Gamma_z^* \Phi_0, \;\;\;\;\; b_z:=\hat\Gamma_z^* d\Phi_0(\nabla H) 
\end{equation}
are two fixed functions on $\RR$ satisfying (by \eqref{EP})
\begin{equation} \label{BASIC}   
b_z+ A_Ta_z \in PW_T(\RR). 
\end{equation}
By \eqref{hztildeEq}, rewrite equation \eqref{VAR} as
\begin{equation} \label{VARB} 
h_z b_z'+ \tilde h_z b_z +A_T(h_z a_z')
=
(h_z b_z)'+A_T(h_z a_z')
\in PW_T(\RR). 
\end{equation}
Our goal is now to show that \eqref{VARB} is impossible for a dense set of $h$.
We denote by $\Hilb:\calS'\ra \calS'$ the Hilbert transform,
defined by
$$
(\Hilb\, f)(t)=-\int_\RR \i\,\sign\,(\xi) e^{\i t\xi}\hat f(\xi) d\xi.
$$
\begin{lem}
\label{HilbertLemma}
Let $h_z$ be given by \eqref{hzEq}. Then 
$
\i h_z b_z+h_za_z'=
2h_z\Gamma_z^\star \dbar\Phi_0(X_H)
$
admits a holomorphic extension to $S_T$.
\end{lem}

\begin{proof}
By \eqref{VARB}, 
$$
\i\xi\,\widehat{h_zb_z}+\xi\coth T\xi\,\widehat{h_za_z'}=o(e^{-T|\xi|}).
$$
Now, $\coth T\xi-\sign(\xi)=O(e^{-2T|\xi})$.
Since $a_z'\in L^\infty(\RR)$ and $h_z/t\in L^\infty(\RR)$, this implies
$$
\i\xi\,\widehat{h_zb_z}+\xi\sign(\xi)\,\widehat{h_za_z'}=o(e^{-T|\xi|}),
$$
or
$\widehat{h_zb_z}-\i\sign(\xi)\,\widehat{h_za_z'}=o(e^{-T|\xi|})$,
i.e., $h_zb_z+\Hilb(h_za_z')\in PW_T$. Since by definition $\Hilb$
maps $PW_T$ to itself, we also have $\Hilb(h_zb_z)-h_za_z'\in PW_T$. Multiply the
former equation by $\i$ and add it to the latter to obtain
$$
(I-\i\Hilb)(\i h_z b_z-h_za_z')\in PW_T,
$$ 
and by conjugation
$(I+\i\Hilb)(\i h_z b_z+h_za_z')\in PW_T$. Since $I+\i\Hilb$
is twice the orthogonal projection operator onto the positive frequency
space, it follows that
$$
\calF(\i h_z b_z+h_za_z')(\xi)=o(e^{-T\xi}), \quad \h{for all $\xi>0$}.
$$
By the proof of \cite[Theorem 3.1]{St} it follows that 
$\i h_z b_z+h_za_z'$ admits a holomorphic extension to the one-sided strip $S_T$.
The lemma now follows from \eqref{FlowLeavesEqs} and \eqref{azbzEq},
and $d\Phi_0(X_H+\i JX_H)=2 d\Phi_0(X^{0,1}_H)=2\dbar\Phi_0(X_H)$.
\end{proof}

\begin{lem}
\label{hzACLemma}
There exist $\alpha<\beta\in\RR$ and $\eps\in(0,T]$ 
all independent of $T$ (but depending on $z$) such that 
$h_z$, given by \eqref{hzEq}, admits a holomorphic extension to 
the two-sided rectangle $[-\eps,\eps]\times [\alpha,\beta]\subset S_T\cup \overline{S_T}$.
In particular, $\hat\Gamma_z^\star h$ is real-analytic on $[\alpha,\beta]$.
\end{lem}

\begin{proof}
First, observe that \eqref{VARB} is true for the constant function $h\equiv 1$ on $M$:
then $(0,(1+\eps)\dot\vp_0,z)\in\calC_{T',z}$
for all $\eps\in[0,\eps_0]$, for some $T'<T$. 
In fact, there exists then a solution to the HCMA \eqref{HCMARayEq}
for some $T'<T$ by reparametrizing $\vp$ in the $s$ variable, and $T'\ra T$
as $\eps_0\ra0$. This proves the claim.

Lemma \ref{HilbertLemma} implies that 
$h_z(\i b_z+a_z')$, respectively $h_z(-\i b_z+a_z')$,
admits a holomorphic extension to $S_T$, respectively, $\overline{S_T}$.
When $h\equiv1$, then $\tilde h_z\equiv 1$ and $h_z(t)=t$, on $\RR$.
Thus, by the previous paragraph, these estimates hold for $h_z=t$.
Hence, $\i b_z+a_z'$, respectively $-\i b_z+a_z'$,
admits a holomorphic extension to $S_T$, respectively, $\overline{S_T}$.
By dividing, and since $h_z$ is real, it follows by the Schwarz
reflection principle that $h_z$ admits a holomorphic extension
to some rectangle $[-\eps,\eps]\times[\alpha,\beta]\subset S_T$ whenever
$\i b_z+a_z'$ does not vanish on $[\alpha,\beta]$ (here we also used
the fact that zeros of holomorphic functions cannot have an accumulation
point).
In particular, $h_z$ is real analytic on $\RR\setminus W_z$, where
$$
W_z:=\{t\in\RR\,:\,a_z'=\Gamma_z^\star X_H\Phi_0(\i t)=
d\Phi_0(X_H)\circ \Gamma_z(\i t)=0\}.
$$

The proof is complete if $\RR\sm W_z$ contains an open interval.
Since $W_z$ is closed it thus suffices to rule out the case where $W_z=\RR$, i.e.,
$a_z'=0$, and hence $a_z=0$, on $\RR$.
If this holds for every point in a neigborhood of $z$
then $\Phi_0$ must be a function of $H$ on some
neighborhood of $z$ in $M$.
Clearly, this is a non-generic property and perturbing either $H$
or adding to $\Phi_0$ the real part of a generic local 
holomorphic function (this does not require changing $\o_{\vp_0}$)  
will destroy this property.
Thus, $h_z$ must be real-analytic at least on some
open interval on $\RR$, and by \eqref{hzEq} and differentiation 
so is $\hat\Gamma_z^\star h$. Now, let 
$[\alpha,\beta]\subset \RR\setminus W_z$ be
any nonempty interval, and note that $W_z$ is independent of $T$.
\end{proof}

We now complete the proof of Theorem \ref{GENERIC}.

By taking $\beta-\alpha$ sufficiently small we may assume
that $\hat\Gamma_z:[\alpha,\beta]\ra M$ is an embedded curve.
Consider the map $R_z:C^3(M)\ra C^3([\alpha,\beta])$ defined
by $R_zf:=\hat\Gamma_z^\star f|_{[\alpha,\beta]}$. Observe
that $R_z$ is a bounded surjective linear operator. Hence,
it defines an open map. Let $B$ be any open ball in $C^3(M)$
containing the zero function $0$.
If for some $T>0$, $\{(\vp_0,\dot\vp_0+f,z)\,:\, f\in B\}\subset \calC_{T,z}$
then Lemma \ref{hzACLemma} implies that $R_z(B)$
is contained in the subset of real-analytic functions in $C^3([\alpha,\beta])$,
with $[\alpha,\beta]$ independent of $T>0$.
However, the latter is not an open subset in $C^3([\alpha,\beta])$.
This concludes the proof of Theorem \ref{GENERIC}.

\section{The smooth lifespan of the HRMA}
\label{LifespanHRMASection}

In this section we restrict to toric manifolds and prove Proposition
\ref{ToricLifespanProp} concerning the analytic continuation of orbits
of Hamiltonian orbits and the invertibility of the associated Moser
maps. The first part, concerning the infinite analytic continuation
of the Hamiltonian flow defined by the Cauchy data, is proved
in Lemma \ref{AnalyticContinuationHamOrbitsToricLemma}. 
The second part, concerning the invertibility of the Moser
maps, is proved in Lemma \ref{InvertibilityMoserMapsToricLemma}.

\subsection{Some background on toric \K manifolds}
\label{ToricBackgroundSubsection}

We briefly recall some background facts on toric \K manifolds.
For more detailled background we refer to \cite{R,RZ1} and references therein.

A symplectic toric manifold is a compact closed
\K manifold $(M,\o)$ whose automorphism group contains a complex
torus $(\CC^\star)^n$ whose action on a generic point is
isomorphic to $(\CC^\star)^n$, and for which the real torus
$(S^1)^n\subset (\CC^\star)^n$ acts in a Hamiltonian fashion by
isometries.

We will work with coordinates on the open dense orbit of the
complex torus given by $z_j=e^{x_j/2+\i\th_j}, j=1,\ldots,n$, with
$(x,\th)=(x_1,\ldots,x_n,\th_1,\ldots,\th_n)\in\R^n\times (S^1)^n$.
Let $M_\open\isom(\CC^\star)^n$ be the open orbit
of the complex torus in $M$ and write
\begin{equation}
\label{OpenOrbitPotentialEq}
\o|_{M_\open}=\i\ddbar\psi_\o.
\end{equation}
We call $\psi_\o$ the open-orbit \K potential of $\o$.
The real torus $(S^1)^n\subset (\CC^\star)^n$
acts in a Hamiltonian fashion with respect to $\o$. 
The image of the moment map $\nabla\psi_\o$ is a convex Delzant polytope
$P\subset\RR^n$ and depends only on $[\o]$ (note that $\o$ only determines
$P$ up to translation; we fix a strictly convex $\psi_\o$ satisfying
(\ref{OpenOrbitPotentialEq}) to fix $P$). We further assume that
this is a lattice polytope. Being a lattice Delzant polytope
means that: (i) at each vertex meet exactly $n$ edges,
(ii) each edge is contained in the set of points $\{p+tu_{p,j}\,:\, t\ge0\}$
with $p\in\ZZ^n$ a vertex, $u_{p,j}\in\ZZ^n$ and
\begin{equation}
\label{UpjEdgesDef}
\h{span}\{u_{p,1},\ldots,u_{p,n}\}=\ZZ^n.
\end{equation}
Equivalently, there
exist outward pointing normal vectors
$\{v_j\}_{j=1}^d\subset\ZZ^n$ that are primitive (i.e., their
components have no common factor) to the $d$ facets in $\partial
P$ and $P$ may be written as
\begin{equation}
\label{PdefEq}
P=\{y\in\RR^n\,:\, l_j(y):=\langle
y,v_j\rangle-\lambda_j\le0,\quad j=1,\ldots,d\},
\end{equation}
with $\lambda_j=\langle p,v_j\rangle\in\ZZ$ with $p$ any vertex on
the $j$-th facet, and $y$ the coordinate on $\RR^n$.

Given a toric metric $\o_\vp$ its corresponding open-orbit \K potential $\psi$ is a strictly
convex function on $\RR^n$ in logarithmic coordinates. Therefore its gradient $\nabla\psi$ is one-to-one onto
$P=\overline{\Im\nabla\psi}$. Its Legendre dual $u:=\psi^\star$, called the symplectic 
potential,
is a strictly convex function on $P$.
Recall the following formulas that will be used throughout 
\begin{equation}
\label
{GradientLegendreTransformEq}
(\nabla\psi)^{-1}(y)=\nabla u(y),
\end{equation}
\begin{equation}
\label{HessianLegendreTransformEq}
(\nabla^2 \psi)^{-1}|_{(\nabla\psi)^{-1}(y)}=\nabla^2 u|_y,
\end{equation}
and if $\eta(s)$ is a one-parameter family of \K potentials and
$u(s):=\eta(s)^\star$ the corresponding symplectic potentials
then
\begin{equation}
\label{VariationPotentialEq}
\dot\eta(s) =-\dot u(s)\circ \nabla\eta(s).
\end{equation}
The proofs of these identities, assuming at least $C^2$ regularity,
can be found in \cite[pp. 84--87]{R}.

\subsection{Complexifying Hamiltonian flows on toric manifolds}
\label{ComplexifyingHamFlowsSubsection}

First we establish the following result regarding the existence of analytic
continuations for the Hamiltonian orbits. It shows that on a toric manifold
any smooth Cauchy data is good, and moreover gives an explicit expression for
the associated Moser maps.

\begin{lem}
\label
{AnalyticContinuationHamOrbitsToricLemma}
Let $(M,J,\o)$ be a toric \K manifold. Given a toric
\K potential $\vp_0$ let $\psi_0$ be a smooth strictly convex function on $\RR^n$
such that over the open orbit $\o_{\vp_0}=\i\ddbar\psi_0$,
and let
$\dot\vp_0$ be a smooth torus-invariant function on $M$.
For every $z\in M_\open$,
the orbit of the Hamiltonian vector field $X_{\dot\vp_0}^{\o_{\vp_0}}$
admits an analytic continuation to the strip $S_\infty$.
Moreover, it is given explicitly by 
\begin{equation} 
f_{\tau}(z) =
\exp-\i\tau X_{\dot\vp_0}^{\o_{\vp_0}}:
z\mapsto z-\tau(\nabla^2\psi_0)^{-1}\nabla_{x}\dot\vp_0,\quad \tau\in S_\infty.
\end{equation}
This expression remains valid on the divisor at infinity if we restrict to the orbit coordinates $\tilde x$
on a slice containing $z$.
\end{lem}

Here (and in similar expressions below) by $(\nabla^2\psi_0)^{-1}\nabla_{x}\dot\vp_0$ we mean the usual
matrix multiplication of the matrix $(\nabla^2\psi_0)^{-1}(x)$
and the vector $\nabla_{x}\dot\vp_0(x)$.

\begin{proof}
The moment coordinates $y$ on the polytope $P$ and the angular coordinates on the regular orbits
are action-angle coordinates for the $(S^1)^n$ Hamiltonian action on $(M,\o_{\vp_0})$, in other words
\begin{equation}
\label{SymplecticFormMomentCoordsEq}
(\nabla\psi_0)_\star\o_{\vp_0}=\sum_{j=1}^n dy_j\w d\th_j,\quad \h{\ over\ \ $(P\setminus\del P)\times(S^1)^n$}.
\end{equation}
The Hamiltonian vector field of $\dot\vp_0$ is given in these coordinates by
\begin{equation}
\label{PushForwardHamVectorFieldEq}
(\nabla\psi_0)_\star X^{\o_{\vp_0}}_{\dot\vp_0}
=
-\sum_{j=1}^n\frac{\del\dot\vp_0}{\del y_j}\big((\nabla\psi_0)^{-1}(y)\big)\frac{\del}{\del\th_j},
\quad y\in P\setminus\del P.
\end{equation}
Therefore the Hamiltonian flow of $X^{\o_{\vp_0}}_{\dot\vp_0}$ is given, in terms of the moment coordinates,
by
\begin{equation}
\label{HamFlowToricPullbackEq}
\nabla\psi_0\circ\exp t X^{\o_{\vp_0}}_{\dot\vp_0}\circ(\nabla\psi_0)^{-1}.(y,\th)
=
(y,\th-t\nabla_y\dot\vp_0\circ(\nabla\psi_0)^{-1}), \quad \h{\ over\ \ $(P\setminus\del
P)\times(S^1)^n$},
\end{equation}
and in terms of the coordinates on $M_\open$ by
$$
\exp t X^{\o_{\vp_0}}_{\dot\vp_0}.(x,\th)
=
(x,\th-t(\nabla^2\psi_0)^{-1}\nabla_x\dot\vp_0).
$$
It therefore admits a holomorphic extension to a map $\exp \i\tau X^{\o_{\vp_0}}_{\dot\vp_0},
\tau=s+\i t$, given in these coordinates by (using (\ref{GradientLegendreTransformEq})-(\ref{HessianLegendreTransformEq}))
\begin{equation}
\label{HamFlowOpenOrbitEq}
\begin{aligned}
\exp -\i\tau X^{\o_{\vp_0}}_{\dot\vp_0}.(x,\th)
& =
(x-s(\nabla^2\psi_0)^{-1}\nabla_x\dot\vp_0\,,\;\th-t(\nabla^2\psi_0)^{-1}\nabla_x\dot\vp_0),
\cr
& 
\qquad\qquad\qquad
\qquad\qquad\qquad
\qquad\qquad
s\in\RR_+,t\in\RR.
\end{aligned}
\end{equation}
For each $z\in M_\open$, this is a holomorphic map of $S_\infty$ into $M_\open\subset M$
since in terms of the complex coordinates $z_j:=x_j+\i\th_j$
it is given by an affine map
$$
\tau\mapsto z-\tau(\nabla^2\psi_0)^{-1}\nabla_{x}\dot\vp_0,\quad \tau\in S_\infty.
$$

It remains to consider orbits of points $z\in M\setminus M_\open$ (for these points
Equation (\ref{HamFlowOpenOrbitEq}) is not valid), and this essentially
amounts to some toric bookkeeping.
Let $F\subset\partial P$ be a codimension $k$ face of $P$
cut out by the equations (see (\ref{PdefEq}))
\begin{equation}
\label{FFaceDefEq}
F:=\{y\in\del P\,:\,\langle y, v_{j_i}\rangle=\lambda_{j_i},\quad i=1,\ldots,k\},
\end{equation}
and assume that $z$ corresponds to a point in the interior of $F$. More precisely,
assume that for a sequence of points $\{z_i\}\subset M_\open$ converging to $z$
the points $\nabla\psi_0(z_i)$ converge to a point in the interior of $F$.
On points in $M$ that correspond to points in $F\setminus \del F$ the stabilizer of
the $(S^1)^n$-action action is $k$-dimensional.
In other words, when restricted to $F\setminus \del F$,
the vector fields $\frac{\del}{\del\th_1},\ldots,\frac{\del}{\del\th_n}$
span an $(n-k)$-dimensional distribution.
Without loss of generality we may assume that in (\ref{FFaceDefEq}) we have
$\{j_i,\ldots,j_k\}=\{1,\ldots,k\}$ (otherwise rename the labels).
Let $p\in\del F$ be a vertex and let $u_{p,1},\ldots,u_{p,n}$ be the vector
defining the edges emanating from $p$, as in (\ref{UpjEdgesDef}). Without
loss of generality assume the vectors $u_{p,1},\ldots,u_{p,n-k}$ span $F$.
On $F\setminus\del F$ the vectors $\{v_1,\ldots,v_k,u_{p,1},\ldots,u_{p,n-k}\}$
span $\RR^n$ and one may find $n-k$ unit vectors $\tilde u_{p,1},\ldots,\tilde u_{p,n-k}$
such that $\{v_1,\ldots,v_k,\tilde u_{p,1},\ldots,\tilde u_{p,n-k}\}$ form an
orthonormal basis. Let $U$ denote the orthogonal matrix obtained
from these $n$ column vectors. Let $\tilde y:=yU$ and $\tilde\th:=\th U$.
Then in these coordinates (\ref{SymplecticFormMomentCoordsEq}) becomes
\begin{equation}
\label{SymplecticFormAdaptedMomentCoordsEq}
(\nabla\psi_0)_\star\o_{\vp_0}=\sum_{j=1}^n d\tilde y_j\w d\tilde\th_j,\quad
\h{\ over\ \ $(P\setminus\del P)\times(S^1)^n$}.
\end{equation}
The advantage of this formula is that it specializes to the following formula
when restricted to $F\setminus \del F$:
\begin{equation}
\label{SymplecticFormAdaptedMomentCoordsRestrictionEq}
\o_{\vp_0}|_{(F\setminus \del F)\times(S^1)^{n-k}}
=
\sum_{j=k+1}^{n} d\tilde y_j\w d\tilde\th_j.
\end{equation}
Hence, in these coordinates the Hamiltonian flow of $\dot\vp_0$ is given by
$$
(\tilde y,\tilde\th)
\mapsto
(\tilde y,\tilde\th_{1},\ldots,\tilde\th_k,
\tilde\th_{k+1}+t\nabla_{\tilde y_{k+1}}\dot u_0,
\ldots,
\tilde\th_{n}+t\nabla_{\tilde y_{n}}\dot u_0).
$$
In order to describe the complexification of this map in $M$,
we use local holomorphic slice-orbit coordinates
$(z',z'')\in\CC^k\times(\CC^\star)^{n-k}$ (see, e.g., \cite{SoZ1})
that can be described as follows. The stabilizer of $(\CC^\star)^n$ at $z$ is $(\CC^\star)^k$.
The tangent space $T_zM$ decomposes to the tangent space to the orbit of $z$, $T_z((\CC^\star)^{n-k}.z)$,
and its normal $(T_z((\CC^\star)^{n-k}.z))^{\perp}$. Intersecting each of these spaces with the unit
ball in $T_zM$ we therefore obtain local holomorphic coordinates $(z',z'')\in\CC^k\times(\CC^\star)^{n-k}$.
The coordinates $z'$ are called the slice coordinates, while the $z''$ are called the orbit coordinates.
We may write $z''_j=\tilde x_j/2+\i\tilde\th_j, \, j=k+1,\ldots,n$, with $\tilde x_j=\nabla u_0(\tilde y_j)$.
On $F\setminus \del F$ the matrix $\nabla_{\tilde y}u$ is of rank $n-k$ with the bottom
$(n-k)\times(n-k)$ block an invertible matrix. 
The same reasoning as before now shows that we have
a formula analogous to (\ref{HamFlowOpenOrbitEq}) 
where we replace $(\nabla^2_x\psi)^{-1}$
by that block of $\nabla^2_{\tilde y}u$,
and $\nabla_x\dot\vp_0$ by $(\nabla_{\tilde x_{k+1}}\dot\vp_0,\ldots,\nabla_{\tilde x_{n}}\dot\vp_0)$.
Once again we see that the resulting maps extend
to the strip $S_\infty$, and this concludes the proof of the Lemma.
\end{proof}

\begin{remark}
\label{ToricNoCommutativityRemark}
{\rm
As the Lemma shows, the Hamiltonian orbits admit an analytic
continuation to the whole upper half plane. In relation to
Remark \ref{NoGroupLawRemark}, we point out that nevertheless
the Moser maps do not generically obey a group law in the holomorphic variable $\tau$.
To see this,  
change variables to the action-angle variables $(y, \theta)$. 
Since $y=y(x)$,
$X^{\o_{\vp_0}}_H = -\sum_j \frac{\partial H}{\partial I_j}
\frac{\partial}{\partial \theta_j}$
and
$J X^{\o_{\vp_0}}_H =  -\sum_j \frac{\partial H}{\partial I_j}
\frac{\partial}{\partial x_j}$ 
(with a slight abuse of notation as compared to \eqref{PushForwardHamVectorFieldEq}).
Then
$$
[X^{\o_{\vp_0}}_H, J X^{\o_{\vp_0}}_H ] 
= -  \sum_{j, k} \frac{\partial H}{\partial I_k}
\frac{\partial^2 H}{\partial x_k \partial I_j}
\frac{\partial}{\partial \theta_j},
$$
vanishing only if the matrix 
$\begin{pmatrix} \frac{\partial^2 H}{\partial x_k \partial I_j} \end{pmatrix}$ 
has a kernel, which is generically false. 
}
\end{remark}

\subsection{Moser flows on toric manifolds}
\label{MoserFlowsSubsection}

Having derived an explicit expression for the analytic continuations of the Hamiltonian
orbits for all imaginary time, we now turn to investigate the invertibility of the
resulting Moser maps.

\begin{lem}
\label
{InvertibilityMoserMapsToricLemma}
Let $(M,J,\o)$ be a toric \K manifold. Given a toric
\K potential $\vp_0$ let $\psi_0$ be a smooth strictly convex function on $\RR^n$
such that over the open orbit $\o_{\vp_0}=\i\ddbar\psi_0$,
and let
$\dot\vp_0$ be a smooth torus-invariant function on $M$.
The Moser maps $f_s(z)=\exp-\i sX_{\dot\vp_0}^{\o_{\vp_0}}.z$ defined by
Lemma \ref{AnalyticContinuationHamOrbitsToricLemma} are smoothly invertible
if and only if
\begin{equation}
\label{ToricTspanSecondEq}
s<
T_\span^\cvx:=
\sup\,\{\,a>0: \psi^\star_0-a\dot \vp_0\circ(\nabla\psi_0)^{-1} \hbox{\rm\ is convex}\}.
\end{equation}

\end{lem}

Note that the formula for $T^\cvx_\span$ is well-defined independently of
the choice of the open-orbit \K potential $\psi_0$ for $\o_{\vp_0}$.

\begin{proof}
From the proof of Lemma \ref{AnalyticContinuationHamOrbitsToricLemma}  (cf.\eqref{HamFlowOpenOrbitEq})  we have the following formula
for the Moser maps, restricted to the open orbit,
\begin{equation}
\label{MoserMapOpenOrbitEq}
f_s(z)=z-s(\nabla^2\psi_0)^{-1}\nabla_{x}\dot\vp_0,\quad \tau\in S_\infty, z\in M_\open,
\end{equation}
or in terms of the moment coordinates
\begin{equation}
\label{MoserMapOpenOrbitMomentCoordEq}
f_s(\nabla u_0(y))
=
\nabla_y u_0(y)
+
s\nabla_{y}\dot u_0,\quad s\in \RR_+, y\in P\setminus\partial P.
\end{equation}
Since $\nabla_y=\nabla^2_yu_0.\nabla_x$,
applying the gradient with respect to $y$ to this equation we obtain
$$
\nabla^2 u_0(y). \nabla_x f_s(\nabla u_0(y))
=
\nabla^2_y(u_0+s\dot u_0).
$$
Since $\nabla^2 u_0$ is invertible for $y\in P\setminus\del P$, it follows that the gradient of
$f_s$ is invertible at $z\in M_\open$ if and only if $u_0+s\dot u_0$ is strictly convex on $P\setminus\del P$.
The analysis for $z\in M\setminus M_\open$ is similar, following the technicalities outlined in the
proof of Lemma \ref{AnalyticContinuationHamOrbitsToricLemma}.
Since by definition $u_0=\psi^\star_0$ and 
using (\ref{GradientLegendreTransformEq})
we obtain (\ref{ToricTspanSecondEq}).
\end{proof}

This concludes the proof of Proposition \ref{ToricLifespanProp}.

\section{Leafwise subsolutions for HRMA}
\label
{OptimalSubsolutionHRMASection}

The toric setting is special
in that first the Moser maps exist for all $s\ge0$, and second
that $\o$ admits a \K potential on the whole open orbit $M_\open$.
As in the discussion below  \eqref{CauchyLaplaceLeaveLocalEq}, the Cauchy problem  takes the following form:
\begin{equation}
\label
{CauchyLaplaceLeaveToricEq}
\begin{aligned}
\Delta\chi_z
& =
0,\quad \h{\ on \ } S_\infty,
\cr
\chi_z(\i t)
& = 
\psi_0\circ f_{\i t}(z),\quad \h{\ on \ } \del S_\infty,
\cr\dis
\frac{\del\chi_z}{\del s}(\i t)
& = 
\dot\vp_0\circ f_{\i t}(z)
-
\nabla_{g_{\vp_0}}\dot\vp_0(\psi_0)\circ f_{\i t}(z), \quad \h{\ on \ } \del S_\infty.
\end{aligned}
\end{equation}

We now turn to proving that the HRMA (\ref{HRMARayEq}) admits a 
unique leafwise subsolution. 

\begin{proof}[Proof of Proposition \ref{OptimalSubsolutionToricProp}]

First, we record some useful formulas for the Moser maps on a toric
manifold. 
They follow from the proof of Lemma \ref{AnalyticContinuationHamOrbitsToricLemma}
by substituting $y=\nabla\psi_0$ in (\ref{MoserMapOpenOrbitMomentCoordEq}) and using
(\ref{GradientLegendreTransformEq}).

\begin{lem}
\label{MoserPathFormulaToricCor}
Let $\psi_s$ be a smooth solution of the HRMA (\ref{HRMARayEq}), and
let $f_s$ denote the associated Moser diffeomorphisms given by
Lemma \ref{AnalyticContinuationHamOrbitsToricLemma}. Then on the open-orbit,
\begin{equation}
\label{SecondFsToricMomentEq}
f^{-1}_s
=
(\nabla\psi_0)^{-1}\circ\nabla\psi_s
=
\nabla u_0\circ(\nabla u_s)^{-1},\quad s\in[0,T^\cvx_\span),
\end{equation}
and if we let $u_s(y)=u_0(y)+s\dot u_0(y)$, then
\begin{equation}
\label{SecondAllTimeFsToricMomentEq}
f_s
=
\nabla u_s\circ(\nabla u_0)^{-1}
,\quad \h{\ all \ } s\ge0.
\end{equation}
These expressions remain valid globally on $M$ if we use the Euclidean gradient
in the orbit coordinates $\tilde x$ along each slice.

\end{lem}

Observe that \eqref{SecondFsToricMomentEq}
and (\ref{SecondAllTimeFsToricMomentEq}) 
are in agreement with Proposition \ref{ToricLifespanProp} (i),(ii), respectively.

Next, we show that each of the Cauchy problems \eqref{CauchyLaplaceLeaveToricEq}
admits a unique global smooth solution.
In the toric setting the harmonic extension to the generalized leaves
of the foliation is especially simple since the initial
conditions are constant on the boundary of the strip.
In particular, the harmonic functions must be linear along the leaves of the foliation.

\begin{lem}
\label{ToricHarmonicAlongLeavesProp}
For every $z\in M_\open$ the Cauchy problem for the Laplace equation
\eqref{CauchyLaplaceLeaveToricEq}
admits a unique smooth solution, given by 
\begin{equation}
\label{LinearSolutionAlongLeavesHRMAEq}
\chi_z(\tau):=
\langle \nabla\psi_0(z),\nabla(u_0+s\dot u_0)\circ\nabla\psi_0(z)\rangle
-
(u_0+s\dot u_0)\circ\nabla\psi_0(z).
\end{equation}
\end{lem}

\begin{proof}
Note first that uniqueness
holds for the Cauchy problem for the Laplace equation on
a half-plane (this can be obtained from a suitable
generalization of Lemma \ref{LaplaceUniquenessLemma} to
the case $T=\infty$).
We claim that a solution to \eqref{CauchyLaplaceLeaveToricEq} 
is given by (\ref{LinearSolutionAlongLeavesHRMAEq}). 
First, $\chi_z$ is linear in $s$ and independent of $t$, hence harmonic.
Moreover,
by \eqref{GradientLegendreTransformEq},
$$
\chi_z(\i t)
=
\langle \nabla\psi_0(z),\nabla u_0\circ\nabla\psi_0(z)\rangle
-
u_0\circ\nabla\psi_0(z)
= u_0^\star(z)=\psi_0(z),
$$
and by \eqref{HessianLegendreTransformEq} and \eqref{VariationPotentialEq},
\begin{equation}
\begin{aligned}
\frac{\del\chi_z}{\del s}(\i t)
& = 
\langle \nabla\psi_0(z),\nabla\dot u_0\circ\nabla\psi_0(z)\rangle
-
\dot u_0\circ\nabla\psi_0(z)
\cr
& = 
\langle\nabla\psi_0(z),-\nabla^2\psi_0.\nabla\dot\vp_0\rangle
+\dot\vp_0(z)
\cr
& =
-g_{\vp_0}(\nabla\psi_0(z),\nabla\dot\vp_0)
+\dot\vp_0(z)
\cr
& =
-\nabla_{g_{\vp_0}}\dot\vp_0(\psi_0)(z)+\dot\vp_0(z).
\end{aligned}
\end{equation}
Finally, observe that
in (\ref{CauchyLaplaceLeaveToricEq}) one may eliminate 
$f_{\i t}$ since the data
$(\psi_0,\dot\vp_0)$ is $(S^1)^n$-invariant.
Thus, $\chi_z$ satisfies the initial conditions.
\end{proof}

\begin{remark}
{\rm
To see how this Lemma fits in with Proposition \ref{LeafwiseProp},
note that $\Phi_z=\gamma_z^\star(\psi_0-\vp_0)$.
Thus, $p_z(t)=-\del_s\Phi_z=d(\psi_0+\vp_0)(\nabla_{g_{\vp_0}}\dot\vp_0)(\Gamma_z(\i t))$,
and $q_z-p_z=\dot\vp_0(\Gamma_z(\i t))-d\psi_0(\nabla_{g_{\vp_0}}\dot\vp_0)(\Gamma_z(\i t))$
is a constant (depending on $z$), and so naturally admits an analytic continuation
to a whole half-plane.
}
\end{remark}

We now turn to proving that the Cauchy problem 
admits a unique leafwise subsolution,
equal precisely to the Legendre transform potential $\psi_L$
given by \eqref{OptimalSubsolutionToricEq}. Note that we use
interchangeably $z$ and $x=\log|z|^2$, as $\psi_L$ is independent
of $\theta$.

To show that (\ref{OptimalSubsolutionToricEq}) defines a leafwise subsolution
on the open orbit 
it suffices to show that for every $z\in M_\open$ the function $F_z^\star \psi_L$
solves the Cauchy problem
(\ref{CauchyLaplaceLeaveToricEq}). 
Now,
$$
\psi_L(s,z)=u_s^\star(z)=\sup_{y\in P}[\langle y,x\rangle - u_s(y)],
$$
with the supremum achieved in at least one point $y$ that is contained in the
set $(\nabla u_s)^{-1}(z)$. It then follows from (\ref{SecondAllTimeFsToricMomentEq})
and \eqref{LinearSolutionAlongLeavesHRMAEq} that 
$F_z^\star \psi_L=u_s^\star\circ f_\tau(z)=\chi_z(s+\i t)$, and
thus by Lemma \ref{ToricHarmonicAlongLeavesProp} $\psi_L$ defines a leafwise subsolution.
This proves the existence part of Proposition \ref{OptimalSubsolutionToricProp} (i).

To prove the existence part of Proposition \ref{OptimalSubsolutionToricProp} (ii),
it suffices to note that every leafwise subsolution for the HRMA
(\ref{HRMARayEq}) on the open orbit gives rise to a global leafwise subsolution for the HCMA
(\ref{HCMARayEq}) by letting
$$
\vp(s+\i t,z)=\psi_L(s,z)-\psi_0(z).
$$
This can be seen as follows.
Note first that according to Lemma \ref{AnalyticContinuationHamOrbitsToricLemma}
the maps $F_z$ are smooth.
Second, note that according to our description of the Moser maps in orbit coordinates,
it follows that $f_\tau$ preserves the interior of each codimension $k$ toric
subvariety of the divisor at infinity $D$. And so, given $z\in D=M\setminus M_\open$,
the condition $F_z^\star(\pi_2^\star\o+\i\ddbar\vp)$ is equivalent to a Cauchy problem
for the Laplace equation, where we now let $N$ be the open toric variety obtained
as the interior of the codimension $k$ toric subvariety containing $z$. This Cauchy problem
then admits a unique smooth global solution, by working in orbit coordinates, as
in Lemma \ref{AnalyticContinuationHamOrbitsToricLemma}.
And since, as already noted, $F_z$ preserves $N$, the harmonicity of
$F_z^\star(\psi_N+\vp)$ implies that $F_z^\star(\pi_2^\star\o+\i\ddbar\vp)=0$ on $N$,
where here $\psi_N$ is a local \K potential for $\o$ on $N$.

Finally, we prove the uniqueness of the leafwise subsolution just constructed.
Let $\eta(\tau,z)$ be another leafwise subsolution. It will suffice to
prove that $\eta=\psi_L$ on the product of $S_\infty$ and the open-orbit $M_\open$.
Observe that by (\ref{SecondAllTimeFsToricMomentEq}),(\ref{LinearSolutionAlongLeavesHRMAEq}), 
and the fact that $\nabla\psi_0:\RR^n\ra P\setminus\del P$ is
an isomorphism we have
$$
\eta(\tau,\nabla u_s(y))=\langle y,\nabla u_s(y)\rangle-u_s(y).
$$
for every $y\in P\setminus \del P$ and $\tau\in S_\infty$. Since $\psi_L$ satisfies the same equation 
and by \cite[Lemma 7.1]{RZ2} $\Im\nabla u_s|_{P\setminus\partial P}=\Im\nabla u_0|_{P\setminus\partial P}=\RR^n$ it
follows that $\eta=\psi_L$.
\end{proof}

\begin{remark}
{\rm
As a by-product, Propositions
\ref{ToricLifespanProp} and \ref{OptimalSubsolutionToricProp}
give an alternative and conceptual proof that the Legendre
transform solves the homogeneous real Monge--Amp\`ere equation.  Of course, these results show considerably
more since they give information for all time, where the Legendre duality breaks down to some
extent. 
As studied in detail in \cite{RZ2}, the leafwise subsolution $\psi$ measures precisely 
the extent to which the Legendre duality breaks down.
}
\end{remark}

\section{HRMA and the Hamilton--Jacobi equation
}
\label{ConeLifespanSection}

We now turn to showing that there exists no admissible $C^1$ weak solution of the IVP for $T>T_\span^\infty$ and establishing the relation between the HRMA and the Hamilton--Jacobi equation.
By a weak solution we mean a solution in the sense of Alexandrov.  

The first step is the observation that any $C^1$ weak solution of HRMA
is a classical solution of a Hamilton--Jacobi equation. Some steps resemble the arguments of
 Proposition 13.1 in  \cite[\S13]{RZ2}.

Recall that the initial Neumann data 
 $\dot\psi_0$  of the HRMA \eqref{HRMARayEq}  is a bounded function on $\RR^n$ obtained by restricting 
the global Neumann data $\dot\vp_0$ on the toric manifold to the open-orbit.

\begin{proof}[Proof of Theorem \ref{HJThm}]

Given the Cauchy data $ (\psi_0, \dot\psi_0)$ of  \eqref{HRMARayEq}, we set
$$\dot u_0:=-\dot\psi_0\circ(\nabla\psi_0)^{-1}. $$

\begin{lem} 
\label
{GMapLemma}
Let $\eta$ be a $C^1$ admissible solution for the HRMA (\ref{HRMARayEq}).
Define the set-valued map,
$$
G:s\in\RR_+\mapsto \Im\,\nabla\eta(\{s\}\times\RR^n)\subset\RR^{n+1}.
$$
Then  $G(s)=G(0)=\{(-\dot u_0(y),y)\,:\, y\in P\sm\del P\}$, for each $s\in[0,T)$.
\end{lem}

\begin{proof}
Since $\eta$ is admissible,
$\nabla_x\eta(\RR^n)=P\sm\del P$.
Thus, $\nabla\eta(\{s\}\times\RR^n)\subset \RR\times (P\sm\del P)$. 
Note that $G(0)$ is the graph of $-\dot u_0$ over $P\sm\del P$. 
We now prove that $G(s)\subset G(0)$. The idea is that $s \to G(s)$ is a continuous set-valued map. 
If $G(s) $ is not contained in $G(0)$, it would sweep out a set of positive  Lebesgue measure in 
$\RR\times (P\sm\del P)\subset\RR^{n+1}$ as $s$ varies, contrary to
the assumption that $\eta$ is a weak solution. 
Since $\eta(s)$ is $C^1$ and strictly convex, its gradient map
$\nabla_x \eta (s):   \times \R^n \to P\backslash \partial P$ is a homeomorphism to its image, i.e.  has a $C^0$ single valued inverse,
and for each $x_0 \in P\sm\del P$,     $(\nabla_x\eta(s))^{-1}:P\sm\del P\ra\RR^n$  maps an open neighborhood 
of $\nabla_x\eta(s, x_0)\in P$ to an open neighborhood $U$ of $x_0$ in $\RR^n$.

 To clarify the picture, consider the diagram:
\begin{equation}\label{DIAGRAM} 
\begin{array}{ccccc}
\{s\} \times \R^n & \mapright{\nabla \eta}{75}
 & G(s) =   \{\dot{\eta}(s,x), \nabla_x \eta(s, x))\}  \subset \R  \times ( P\sm\del P)  
\\ & & \\
\mapdown{\pi}&  &  \mapdown{\pi}\\ & & \\
\R^n    & \mapleft{(\nabla_x \eta(s))^{-1}}{75} &  P\sm\del P
\end{array}
\end{equation}
Here, $\pi$ is the natural projection. Since $\nabla_x \eta(s, x): \R^n \to P\sm\del P$ is a
homeomorphism, also   $\pi: G(s) \to   P\sm\del P$ is a homeomorphism. Thus, $G(s)$ is a graph over 
$P\sm\del P$.

Now suppose that there exists  $z=(\dot\eta(s,x_0),\nabla_x\eta(s,x_0))\in G(s)\setminus G(0)$, i.e.
 $\dot\eta(s,x_0)\not=-\dot u_0\circ\nabla_x\eta(s,x_0)$. 
Then  $\pi^{-1}(U) \subset G(s)$ is a graph  passing through $z\not\in G(0)$. Since $G(s)$ is a continuous
set-valued mapping and $\pi: G(0) \to U$ is a different graph than $\pi: G(s) \to U$, the intermediate graphs
$\pi: G(\sigma) \to U$ for $\sigma \in [0, s]$ 
must fill out the region in between the graphs and create a set of positive Lebesgue measure. 
To be more precise, put  $S:=\sup\{\sigma\,:\, G(\sigma) \h{ contains } 
(-\dot u_0\circ\nabla_x\eta(s,x),\nabla_x\eta(s,x))\}\le s$.
Again by continuity, $\RR\times U\cap \Big(\cup_{\sigma\in[S,s]}G(\sigma) \Big)$
must contain a set of positive Lebesgue measure in $\RR^{n+1}$. This is impossible,
though, by Definition \ref{AlexandrovDef}. Thus, we have shown that
$G(s)\subset G(0)$ 

But since the projection of $G(s)$ onto the $\RR^n$ factor
equals $P\sm\del P$ for each $s$, and $G(0)$ is a graph over
$P\sm\del P$, the containment just proved implies
the equality $G(s)=G(0)$.
\end{proof}

Thus, by Lemma \ref{GMapLemma} and the differentiability assumption, for each $(s,x)\in[0,T]\times \RR^n$ there exists a unique
$y\in P\setminus\del P$  
such that
$$
\Big(\frac{\del\eta}{\del s}(s,x),\nabla_x\eta(s,x)\Big)=(-\dot u_0(y),y),
$$
or, in other words,
\begin{equation}
\label{DotEtaEq}
\frac{\del\eta}{\del s}(s,x)=
-\dot u_0\circ \nabla_x\eta(s,x),
\end{equation}
which concludes the proof of one direction of Theorem \ref{HJThm}.

For the converse, suppose that $\eta\in C^1([0,T]\times\RR^n)$ is
a solution of the Hamilton--Jacobi equation \eqref{HJEq}.
Then $\Im\nabla\eta\subset G(0)$, and since $G(0)$ has zero Lebesgue
measure in $\RR^{n+1}$,  $\eta$ is a weak
solution of the HRMA.
\end{proof}

\begin{remark}
{\rm
The proof can be generalized to handle admissible solutions
that are only partially $C^1$ regular in the sense of \cite[\S10]{RZ2}.
}
\end{remark}

\begin{proof}[Proof of Proposition \ref{ConeadmissibleIsOptimalProp}]

Let $\psi_L$ denote the leafwise subsolution of the HRMA (\ref{HRMARayEq}) given by
Proposition \ref{OptimalSubsolutionToricProp} and
let $\eta$ be a $C^1$ admissible solution of (\ref{HRMARayEq}) 
(see Definition \ref{AdmissibleDef}). Both $\psi$ and $\eta$ are convex
functions on $[0,T]\times\RR^n$. By Theorem \ref{HJThm} both $\psi$
and $\eta$ are solutions of the Hamilton--Jacobi equation \eqref{HJEq}.
The method of characteristics implies that $C^1$ solutions of \eqref{HJEq}
are unique as long as the characteristics of the equation do not intersect
each other. The equation for the projected characteristic curves ${\bf x}(s)$ is
(see, e.g., \cite[Chapter 3]{Evans})
$$
\dot {\bf x}(s)=\big(1,\nabla_\xi\dot u_0({\bf p}_\xi(s))\big), \qquad {\bf x}(0)=(0,x_0),
$$
while $z(s)$, the solution at ${\bf x}(s)$, satisfies
$$
\dot z(s)=\big(1,\nabla_\xi\dot u_0({\bf p}_\xi(s))\big)\cdot\big(p_\sigma(s),{\bf p}_{\xi}(s)\big), \qquad z(0)=\psi_0(x_0),
$$
and ${\bf p}(s)=(p_\sigma(s),{\bf p}_{\xi}(s))$, the gradient of the solution at ${\bf x}(s)$, satisfies
$$
\dot {\bf p}(s)=0,\qquad {\bf p}(0)=(\dot\psi_0(x_0),\nabla\psi_0(x_0)).
$$
Therefore, ${\bf x}(s)=\big(s,x_0+s\nabla \dot u_0(\nabla\psi_0(x_0))\big)$.
Thus, the projected characteristic do not intersect as long as the map
$(s,x)\mapsto (s,x+s\nabla \dot u_0(\nabla\psi_0(x)))$ is invertible,
or equivalently as long as 
$$
x\mapsto
\nabla u_0\circ\nabla\psi_0(x)+s\nabla\dot u_0\circ\nabla\psi_0(x)
$$
is invertible on $\RR^n$; this is precisely as long as $\nabla u_0+s\nabla\dot u_0$ is
invertible on $P\sm\del P$,
or as long as $u_0+s\dot u_0$ is strictly convex, i.e., precisely for $s<\Tspancvx$. Thus
$\eta=\psi_L$ for $s\le\Tspancvx$. In fact, the equation for ${\bf x}(s)$ shows
that the characteristics for the Hamilton--Jacobi equation precisely coincide with
the leaves of the HRMA foliation. Moreover, the equation for $z(s)$ shows that
$$
\begin{aligned}
z({\bf x}(s))
&=\psi_0(x_0)+s\dot\psi_0(x_0)+s\langle \nabla\dot u_0\circ\nabla\psi_0(x_0),
\nabla\psi_0(x_0)\rangle
\cr
&=-u_0\circ(\nabla\psi_0(x_0))
+\langle \nabla\dot u_0\circ\nabla\psi_0(x_0),
\nabla\psi_0(x_0)\rangle
\cr
&\quad +s\dot\psi_0(x_0)+s\langle \nabla\dot u_0\circ\nabla\psi_0(x_0),
\nabla\psi_0(x_0)\rangle
\cr
&=-u_0\circ(\nabla\psi_0(x_0))
+\langle \nabla u_0\circ\nabla\psi_0(x_0),
\nabla\psi_0(x_0)\rangle
\cr
&\quad -s\dot u_0(\nabla\psi_0(x_0))+s\langle \nabla\dot u_0\circ\nabla\psi_0(x_0),
\nabla\psi_0(x_0)\rangle
\cr
&=-(u_0+s\dot u_0)\circ(\nabla\psi_0(x_0))
+\langle \nabla(u_0+s\dot u_0)\circ\nabla\psi_0(x_0),
\nabla\psi_0(x_0)\rangle
,
\cr
\end{aligned}
$$
while from the equation for ${\bf x}(x)$ we have 
$$
{\bf x}(s)
=
(\nabla u_0+s\nabla \dot u_0)\circ\nabla\psi_0(x_0).
$$
Altogether, letting $u_s:=u_0+s\dot u_0$,  we have
$$
z(\nabla u_s(y))=-u_s(y)+\langle \nabla u_s(y),y\rangle,
$$
or in other words, $z(s,x)=u_s^\star(x)=\psi_L(s,x)$.
\end{proof}

Note that in the proof above we show in essence that
any $C^1$ solution of the HRMA is given by the Hopf--Lax
formula \cite{Hopf,Evans}.

Finally, we relate the orbits of the Moser map, the Hamiltonian
orbits, and the characteristics
in $\R^{n+1}$ of the HRMA. 
The following generalizes to weak $C^1$ solutions of the HRMA 
the well-known `conservation law' \eqref{ConservationLawHCMAEq} of smooth solutions of the HCMA. 

\begin{prop} \label{FTPHIDOT} Let $\eta$ be a $C^1$ weak solution of the HRMA \eqref{HRMARayEq},
and let $\vp=\eta-\psi_0$, considered as a function $M$.
Also, let $f_s$ be the Moser maps $f_s(z)=\exp-\i sX_{\dot\vp_0}^{\o_{\vp_0}}.z$ defined in \eqref{MAPS}
and  Proposition 
\ref{ToricLifespanProp}.
Then
$$
\dot{\vp}_s\circ f_s = \dot{\vp}_0.
$$
Further, the $f_s$-orbits $(s, f_s(x))$ are the leaves of the real Monge--Amp\`ere foliation,
namely the projected characteristics of the Hamilton--Jacobi equation \eqref{HJEq}.
\end{prop}

\begin{proof}
By  combining \eqref{DotEtaEq}, \eqref{SecondAllTimeFsToricMomentEq}
and Propositions \ref{ToricLifespanProp} and \ref{ConeadmissibleIsOptimalProp}, one sees that
this  equation
is equivalent to the Hamilton--Jacobi equation in Theorem \ref{HJThm}.

To prove the last statement we note that the leaves of the  Monge--Amp\`ere foliation are orbits
of the complexified Hamiltonian action $\exp t X_{\dot{\vp}_0}^{\omega_{\vp_0}}$. The real orbits
lie on the orbits of the Hamiltonian $(S^1)^n$-action and the real slice of this torus orbit is a point. 
Hence the real slice is the imaginary time orbit, i.e., the orbit of $f_{s}$.
\end{proof}

\bigskip\bigskip
\noindent {\bf Acknowledgments.}
This material is based upon work supported in part by a NSF
Postdoctoral Research Fellowship and
grants DMS-0603850, 0904252.

\end{document}